\documentclass[12pt]{article}
\usepackage[utf8]{inputenc}
\usepackage{amsthm}
\usepackage{amsmath}
\usepackage{bbm}
\usepackage{amsfonts}
\usepackage{amssymb}
\usepackage[left=2.5cm,right=2.5cm,top=2cm,bottom=2cm]{geometry}

\usepackage{tabto}
\TabPositions{2 cm, 4 cm, 6 cm, 8 cm}

\usepackage{tikz-cd}

\usepackage{tikz}
\usetikzlibrary{arrows, automata, positioning}

\usepackage{tocloft}
\usepackage{appendix}

\usepackage{enumerate}

\usepackage{hyperref}

\newcommand{\Manoa}{M\=anoa }
\newcommand{\Hawaii}{Hawai\kern.05em`\kern.05em\relax i }
 
\newtheorem{theorem}{Theorem}[section]
\newtheorem{corollary}[theorem]{Corollary}
\newtheorem{lemma}[theorem]{Lemma}
\newtheorem{claim}[theorem]{Claim}
\newtheorem{proposition}[theorem]{Proposition}

\theoremstyle{definition}
\newtheorem{definition}[theorem]{Definition}
\newtheorem{example}[theorem]{Example}
\newtheorem*{definition*}{Definition}
\newtheorem{remark}[theorem]{Remark}

\newtheorem*{lemma*}{Lemma}
\newtheorem*{proposition*}{Proposition}
\newtheorem*{theorem*}{Theorem}
\newtheorem*{corollary*}{Corollary}
\newtheorem*{claim*}{Claim}

\theoremstyle{definition}

\newtheorem{question}[theorem]{Question}

\newcommand{\im}{\operatorname{im}}
\newcommand{\Hom}{\operatorname{Hom}}
\newcommand{\Homeo}{\operatorname{Homeo}}

\newcommand{\id}{\operatorname{id}}

\newcommand{\HH}{\operatorname{H}}
\newcommand{\EH}{\operatorname{EH}}
\newcommand{\CC}{\operatorname{C}}
\newcommand{\ZZ}{\operatorname{Z}}
\newcommand{\BB}{\operatorname{B}}

\newcommand{\cl}{\operatorname{cl}}
\newcommand{\scl}{\operatorname{scl}}
\newcommand{\rot}{\operatorname{rot}}
\newcommand{\sv}{\operatorname{SV}}

\begin{document}

\title{Second bounded cohomology of groups acting on $1$-manifolds and applications to spectrum problems}
\author{Francesco Fournier-Facio and Yash Lodha}
\date{\today}
\maketitle

\begin{abstract}
We prove a general criterion for the vanishing of second bounded cohomology (with trivial real coefficients) for groups that admit an action satisfying certain mild hypotheses.
This leads to new computations of the second bounded cohomology for a large class of groups of homeomorphisms of $1$-manifolds, and a plethora of applications. First, we demonstrate that the finitely presented and nonamenable group $G_0$ constructed by the second author with Justin Moore satisfies that every subgroup has vanishing second bounded cohomology. This provides the first solution to a homological version of the von Neumann--Day Problem, posed by Calegari.
Next, we develop a technical refinement of our criterion to demonstrate the existence of finitely generated non-indicable (even simple) left orderable groups with vanishing second bounded cohomology. This answers Question 8 from the 2018 ICM proceedings article of Navas.
Then we provide the first examples of finitely presented groups whose spectrum of stable commutator length contains algebraic irrationals, answering a question of Calegari. 
Finally, we provide the first examples of manifolds whose simplicial volumes are algebraic and irrational, as further evidence towards a conjecture of Heuer and L\"{o}h.
\end{abstract}

\section{Introduction}
Two intricately related invariants that appear in modern group theory are \emph{stable commutator length} and \emph{bounded cohomology}.

Stable commutator length (usually referred to as $\scl$) is a real-valued invariant of groups that allows for algebraic, topological, and analytic descriptions, and has several applications in topology and geometry \cite{calegari}. 
Consider a group $G$. The \emph{commutator length} of an element $g$ in the commutator subgroup $G'$, denoted as $\cl(g)$, is the smallest $n\in \mathbf{N}$ such that $g$ can be expressed as a product of at most $n$ commutators of elements in $G$. Then we define the \emph{stable commutator length} of $g$ as $\scl(g) := \lim_{n\to \infty} \frac{\cl(g^n)}{n}$.
If there exists $k \geq 1$ such that $g^k \in G'$, then we set $\scl(g) := \frac{\scl(g^k)}{k}$. Otherwise, we set $\scl(g) := \infty$.
Since the function given by $n\to \cl(g^n)$ is subadditive, this limit always exists. However, in general, it is hard to compute.
Given a group $G$, we denote $\scl(G) := \{\scl(g)\mid g\in G\}$.

\emph{Bounded cohomology} of groups is defined analogously to standard group cohomology, but taking the topological dual of the simplicial resolution instead of the algebraic dual. This invariant was introduced by Johnson and Trauber in the context of Banach algebras \cite{johnson}, who proved that it vanishes in all positive degrees for amenable groups. Since then, it has become a fundamental tool in several fields \cite{BC, monod}, most notably the geometry of manifolds \cite{gromov} and rigidity theory \cite{rigidity}.
Given a group $G$, its second bounded cohomology with trivial real coefficients $\HH^2_b(G; \mathbf{R})$ is of particular interest in the context of rigidity results. For instance, it is related in a strong way to actions on the circle \cite{ghys}, and a vanishing result for high-rank lattices \cite{lattices, rigidity} leads to a proof of superrigidity for mapping class groups \cite{mcg}.
The vanishing of the second bounded cohomology is often viewed as a weaker form of amenability.

The two invariants are intimately related.
The connection is given by a fundamental result called Bavard Duality \cite{bavard}. This allows one to compute stable commutator length by means of \emph{quasimorphisms}, which represent certain classes in $\HH^2_b(G; \mathbf{R})$. A special case of this states that the stable commutator length of $G$ vanishes identically on $G'$ if and only if a natural map $\HH^2_b(G; \mathbf{R}) \to \HH^2(G; \mathbf{R})$ is injective. \\

In this paper we address a family of closely interconnected questions concerning stable commutator length and bounded cohomology.
Our first result establishes a rather general criterion that allows one to compute the second bounded cohomology of several transformation groups. To state our criterion, we need the following definition.

\begin{definition}
\label{def:cc}
Let $G$ be a group. We say that $G$ has \emph{commuting conjugates} if for every finitely generated subgroup $H \leq G$, there exists an element $f \in G$ such that $[H, H^{f}] = \id$.
%one of the following equivalent conditions holds: 
%\begin{enumerate}
%\item For every finitely generated subgroup $H \leq G$, there exists an element $f \in G$ such that $[H, H^{f}] = id$.
%\item For every finitely generated subgroup $H \leq G$, there exist a set of elements $\{f_n\}_{n\in \mathbf{N}} \subseteq G$ such that $[H^{f_i}, H^{f_j}] = id$ whenever $i\neq j\in \mathbf{N}$.
%\end{enumerate}
\end{definition}

\begin{theorem}
\label{thm:main}
Let $G$ be a group with commuting conjugates. Then $\HH^2_b(G; \mathbf{R}) = 0$.
\end{theorem}

Theorem \ref{thm:main} has several nice consequences for the group $G$. It implies that that the stable commutator length vanishes on all of $G'$\footnote{After a talk about this work, Jonathan Bowden pointed out to us that the result establishing the vanishing of scl was already known \cite[Proposition 2.2]{kotschick}.}, that every minimal action on the circle is conjugate to an action by rotations \cite{matsumoto} (see also \cite[Corollary 10.28]{BC}), as well as further rigidity results that we will soon mention. \\

This criterion applies to large families of boundedly supported transformation groups. Together with known combination results, we obtain the following theorem, which is the starting point of our main applications:

\begin{theorem}\label{thm:PLPP}
We have $\HH^2_b(G; \mathbf{R}) = 0$, when $G$ is one of the following:
\begin{enumerate}
\item A group of orientation-preserving piecewise linear homeomorphisms of the interval, or piecewise projective homeomorphisms of the real line;

\item A chain group of homeomorphisms (as defined in \cite{chain}), or a group that admits a coherent action on the real line (as defined in \cite{coherent}).
\end{enumerate}
\end{theorem}

As a first application, we settle a \emph{homological von Neumann--Day problem}: \emph{Do there exist nonamenable groups, all of whose subgroups have vanishing second bounded cohomology (with trivial real coefficients)?} In \cite{scl_pl}, Calegari states a related conjecture which asserts: \emph{If $G$ is a finitely presented, torsion free group with the property that every subgroup has vanishing stable commutator length, then $G$ is amenable.} Applying Theorem \ref{thm:PLPP} to the nonamenable group of piecewise projective homeomorphisms $G_0$ constructed by the second author with Moore \cite{LodhaMoore} disproves the conjecture as well as settles the problem.

%\begin{theorem}\label{thm: homologicalVNDay}
%There exists a finitely presented (and type $F_\infty$), torsion free, nonamenable group $G_0$ such that for every subgroup $H\leq G_0$, it holds that $\HH_b^2(H;\mathbf{R})=0$.
%\end{theorem}

We wish to emphasise that Theorem \ref{thm:PLPP} is a new result even for the special case of subgroups of Thompson's group $F$.
As a consequence, it provides a new obstruction towards the embeddability of a given group in Thompson's group $F$, and more generally in the group of piecewise linear (or piecewise projective) homeomorphisms of the interval (real line).
The subgroup structure of these groups is very mysterious (see e.g. \cite{subgroups1, subgroups2, subgroups3}), and our theorem provides new insight in this direction.
We also remark that as corollaries of Theorem \ref{thm:PLPP}, we recover the theorem of Calegari which asserts that $\scl$ vanishes for any group of piecewise linear homeomorphisms of the unit interval \cite{scl_pl}, as well as the Brin--Squier Theorem which asserts that such groups do not contain nonabelian free subgroups \cite{BrinSquier}.
To our knowledge, these were the only meaningful structural results that were previously known to hold for \emph{all} piecewise linear groups of homeomorphisms of the interval.\\

Our next application is in the theory of left orderable groups \cite{orders}. A group $G$ is said to be \emph{left orderable} if it admits a total order which is invariant under left multiplication.
This notion has a beautiful connection with dynamics of group actions on the line: a countable group $G$ is left orderable if and only if it admits a faithful action by orientation-preserving homeomorphisms on the real line. Two closely related algebraic notions are the following. A group is \emph{indicable} if it admits a surjection onto $\mathbf{Z}$, and \emph{locally indicable} if every finitely generated subgroup is indicable. Note that for a finitely generated group, local indicability implies indicability.

It is well known that locally indicable groups are left orderable, yet the converse fails. However, the converse holds in certain situations.
A fundamental theorem in this direction is Witte-Morris's Theorem \cite{WM}: amenable left orderable groups are locally indicable. Equivalently, finitely generated amenable left orderable groups are indicable. 
In his $2018$ ICM proceedings article \cite{questions}, Andr\'{e}s Navas uses this theorem as a starting point for a research program consisting of a list of questions about left orderable groups which are reminiscent of the Tits alternative. The general theme is: Are some of the properties weaker than amenability enough to imply indicability of a finitely generated left orderable group?
%In this paper, we answer the following question from Navas's problem list:

\begin{question}[{\cite[Question 8]{questions}}]
\label{qNavas}
Does there exist a finitely generated, non-indicable, left orderable group $G$ such that $\HH^2_b(G; \mathbf{R}) = 0$?
\end{question}

There are two key difficulties in attempting to resolve Question \ref{qNavas}. 
The first is the hypothesis of finite generation. Indeed, using the construction of Mather \cite{Mather} it is easy to show that every  countable left orderable group embeds in a countable left orderable perfect group which has vanishing bounded cohomology in every positive degree \cite{bcbinate}. But such groups are not finitely generatable.
When restricting to finitely generated groups, a further difficulty in Question \ref{qNavas} lies in the requirement that the finitely generated group must be \emph{non-indicable}, which for finitely generated groups is stronger that \emph{non-locally indicable}. In other words, the witness to the failure of local indicability is required to be $G$ itself, and not some finitely generated subgroup thereof. Indeed, non-locally indicable groups with vanishing second bounded cohomology are easily constructed using Theorem \ref{thm:PLPP}, and the fact that every finitely generated left orderable group embeds into a chain group \cite[Theorem 1.4]{chain}. Yet, they do not answer Question \ref{qNavas}.

In \cite{grho} the second author and Hyde constructed the first family of finitely generated simple left orderable groups. The construction takes as input a so-called \emph{quasi-periodic labelling} $\rho$ and outputs a finitely generated simple left orderable group $G_\rho$ (see Section \ref{s:grho} for more detail). Note that the finitely generated groups $G_{\rho}$ are non-indicable, being simple. Thus our theorem below shows that the groups $G_{\rho}$ provide a positive answer to Question \ref{qNavas}:

\begin{theorem}
\label{theorem:mainNavas}
Let $\rho$ be a quasi-periodic labelling. Then $\HH^2_b(G_{\rho}; \mathbf{R}) = 0$.
\end{theorem}

We wish to emphasise here that our solution to Navas's question is even more striking owing to the fact that the groups $G_{\rho}$ are simple, which is much stronger than being non-indicable. 
The proof of Theorem \ref{theorem:mainNavas} relies on a technical refinement of the ingredients of Theorem \ref{thm:main}, since as such it is far from being directly applicable to $G_{\rho}$.
We also recover the fact that the groups $G_\rho$ are \emph{left orderable monsters}, meaning that every faithful action on the real line is globally contracting \cite[Corollary 0.3]{uniformlyperfect}. We are also able to deduce the stronger fact that second bounded cohomology vanishes with integral coefficients, which in turn implies that every action on the circle has a global fixpoint, strengthening \cite[Corollary 0.2]{uniformlyperfect} (Corollary \ref{cor:integral}).

To prove Theorem \ref{theorem:mainNavas}, we will start by applying Theorem \ref{thm:main} in order to show that subgroups $\Gamma \leq G_{\rho}$ with a global fixpoint satisfy $\HH^2_b(\Gamma; \mathbf{R}) = 0$. The promotion of these computations to $G_{\rho}$ is non-trivial, thus showing how Theorem \ref{thm:main} can be useful also as a technical ingredient to provide more complicated computations of second bounded cohomology. As a further illustration of this, we will sketch how our arguments may be adapted to prove vanishing of second bounded cohomology of the groups $T(\varphi, \sigma)$ from \cite{LBMB}, which are also examples of finitely generated simple left orderable groups, and thus also answer Question \ref{qNavas}. \\

The focus of this paper is on groups acting on $1$-manifolds. However, the criterion of Theorem \ref{thm:main} applies in much greater generality.
For instance, the authors and Matthew C.B. Zaremsky used Theorem \ref{thm:main} to establish the vanishing of second bounded cohomology of certain braided Thompson groups, via their actions on Cantor set complements in the plane \cite{bV}. As another example, Theorem \ref{thm:main} has been used by the first author, L\"oh and Moraschini to show that the direct limit general linear group of the suspension over a ring, a group which is central to algebraic $K$-theory, has vanishing second and third bounded cohomology \cite[Proposition 4.9]{bcbinate}. The latter result showcases how our criterion may be combined with cohomological techniques to prove vanishing results in degree $3$ as well.
The method of proof of Theorem \ref{theorem:mainNavas} also applies beyond actions on $1$-manifolds: in Subsection \ref{ss:nico} we will see an example of application to actions on solenoids. \\

%Monod and Nariman recently proved that the group $\Homeo_0(S^n)$ of homeomorphisms isotopic to the identity, has vanishing second and third bounded cohomology, for $n = 2, 3$ \cite[Theorem 7.1]{MonodNariman}. Their proof uses cohomological techniques to reduce to the fact that the groups $\Homeo_c(S^{n-1} \times \mathbf{R})$ of compactly supported homeomorphisms have vanishing second bounded cohomology. It is easy to see that $\Homeo_c(S^{n-1} \times \mathbf{R})$ has commuting conjugates for every $n \geq 2$, and thus the same proof implies that $\Homeo_0(S^n)$ has vanishing second and third bounded cohomology for every $n \geq 2$.

Moving to actions on the circle, applying recent work on boundedly acyclic resolutions \cite{bcbinate} we are able to deduce the following result from Theorem \ref{thm:PLPP}:

\begin{theorem}
\label{thm:T}
Let $G$ be a group of orientation-preserving piecewise linear or piecewise projective homeomorphisms of the circle. Suppose that $G$ has an orbit $Y$ such that the action on circularly ordered pairs and triples in $Y$ is transitive. Then $\HH^2_b(G; \mathbf{R})$ is one-dimensional, spanned by the real Euler class.
\end{theorem}

Our result is a strong generalization of the corresponding statement for the special case of Thompson's group $T$ \cite{spectrum1}. In particular, it shows that many Stein--Thompson groups have second bounded cohomology generated by the real Euler class, which was conjectured by Heuer and L\"oh \cite[Conjecture A.5]{spectrumv1}. This also applies to a multitude of other examples, including the finitely presented infinite simple group $S$ constructed by the second author in \cite{lodhaS}. See \cite{monsters} for a detailed discussion.

This result has interesting consequences for rigidity of group actions.
Using the interpretation of actions on the circle via bounded cohomology, we deduce that if $G$ is as above and the given action is proximal (see Subsection \ref{ss:dynamics} for the definition), then $G$ admits a unique such action on the circle up to conjugacy (Corollary \ref{cor:rigidity}). This fact was known for Thompson's group $T$ and certain related groups, but our result provides a considerable generalization to a much larger family.

More generally, the finite-dimensionality of the second bounded cohomology of the groups from Theorems \ref{thm:main}, \ref{thm:PLPP} and \ref{thm:T} implies further rigidity results. For instance, every homomorphism from $G$ to the mapping class group of a hyperbolic surface has virtually abelian image \cite{mcg}, and if $G$ is moreover finitely generated, then it admits only finitely many conjugacy classes of isometric actions on an irreducible symmetric space which is Hermitian and not of tube type \cite{hermitian1, hermitian2}. \\

The next application concerns the \emph{spectrum of $\scl$ on finitely presented groups} which is defined as $\{r\in \mathbf{R}_{\geq 0} \mid \exists \text{ a finitely presented group }G\text{ such that }r\in \scl(G)\}$.
%A question that has received recent attention is to compute this on certain classes of groups. 
For the classes of finitely generated and recursively presented groups, the spectrum has been completely described by Heuer in \cite{heuer}. However the full picture for finitely presented groups remains mysterious \cite[Question 5.48]{calegari}. It has been shown for several classes of finitely presented groups that the spectrum of $\scl$ consists entirely of rationals. This holds for instance for the lift of Thompson's group $T$
\cite{ghys_serg, zhuang}, free groups \cite{free} and certain graphs of groups \cite{gog}.

In \cite[6 C]{gromov}, Gromov asked whether irrational values can occur in the scl of finitely presented groups. The first examples were found by Zhuang \cite{zhuang}, who showed that lifts of certain Stein--Thompson's groups have elements with transcendental $\scl$. Since then other examples have emerged \cite[Chapter 5]{calegari}, but the $\scl$ always appeared to be either rational or transcendental. This led to the following natural question:

\begin{question}[Calegari {\cite[Question 5.47]{calegari}}]
\label{qCalegari}

Is there a finitely presented group in which $\scl$ takes on an irrational value that is algebraic?
\end{question}

This question was also asked in \cite{book} and mentioned in \cite{heuer}. Here we answer it in the positive:

\begin{theorem}
\label{thm:scl}
There exists a finitely presented (and type $F_\infty$) group $G$ such that $\scl(G)$ contains all of $\frac{1}{2} \mathbf{Z}[\tau]$, where $\tau = \frac{\sqrt{5} - 1}{2}$ is the small golden ratio.
\end{theorem}

The group $G$ is the lift to the real line of the \emph{golden ratio Thompson's group} $T_\tau$, which is the natural ``circle version" of the group $F_{\tau}$ defined by Cleary in \cite{cleary}. 
(Just like Thompson's group $T$ is the ``circle version" of Thompson's group $F$.)
The group $T_{\tau}$ differs from $T$ for two main reasons: it contains irrational rotations (while every element in $T$ has a periodic point) and its commutator subgroup has index two and is simple (while $T$ itself is simple).

The proof of Theorem \ref{thm:scl} consists in analyzing the structure of $T_\tau$ to show that it satisfies the hypotheses of Theorem \ref{thm:T}, after which the computation of the stable commutator length of certain elements in the lift will follow from our results on groups acting on the circle. \\

% We also give another, more elementary proof of Theorem \ref{thm:scl} that does not make use of the full power of Theorem \ref{thm:T}. This is a  dynamical approach that also uses some of the combinatorial analysis carried over in \cite{burilloT}. On the way, we prove uniform simplicity of $T_\tau'$, which is of independent interest and strengthens the simplicity proven in \cite{burilloT}. \\

The final application concerns the \emph{spectrum of simplicial volume}. The notion of simplicial volume $\| M \|$ of an (oriented, closed, connected) manifold $M$ was introduced by Gromov in \cite{gromov}, and it was the main motivation to extend the definition of bounded cohomology to topological spaces.
It is defined as the simplicial norm of the fundamental class of the manifold, and provides a homotopy invariant that captures the topological complexity of the manifold and admits a plethora of applications in geometry \cite[Chapter 7]{BC}.

A question that has received recent attention is to understand the spectrum $\sv(d)$ of the simplicial volume for oriented, closed, connected manifolds in dimension $d>3$.
A recent breakthrough result of Heuer and L\"{o}h proves that $\sv(d)$ is dense in $\mathbf{R}_{\geq 0}$ for each $d>3$ and that $\mathbf{Q}_{\geq 0} \subset \sv(4)$ \cite{spectrum1}. Moreover, $\sv(4)$ also contains transcendental values \cite{spectrum2}. %Their main tool is a general result that allows one to promote values of stable commutator length of a finitely presented group to values of simplicial volume, so also in this case all known values in $\sv(d)$ were either rational or transcendental. 
In this paper, we are able to obtain the first irrational algebraic values:

\begin{theorem}
\label{thm:sv}
There exists a manifold $M$ such that the simplicial volume $\| M \|$ is algebraic and irrational.
\end{theorem}

More concretely, we demonstrate that for every real $x\in 48 \cdot \mathbf{Z}[\tau]$ (where $\tau = \frac{\sqrt{5}-1}{2}$ is the small golden ratio),
there is an oriented, closed, connected $4$-manifold $M$ such that $\| M \| = x$.
We stress that the general results from \cite{spectrum1} and \cite{spectrum2} only allow to promote values of stable commutator length to values of simplicial volume under a full understanding of the second bounded cohomology. In particular, Theorem \ref{thm:sv} is not a direct consequence of Theorem \ref{thm:scl}, and it relies on the full power of Theorem \ref{thm:T}. In fact, the same method shows that several known trascendental values of stable commutator length can be promoted to simplicial volumes, providing new evidence toward the conjecture that the two spectra coincide \cite[Question 1.1]{spectrum2}.

\begin{remark}A few months after the first version of this paper, Monod proved that the bounded cohomology of Thompson's group $F$ vanishes in all positive degrees, and including some non-trivial coefficients \cite{lamplighters}. The proof in \cite{lamplighters} makes crucial use of \emph{self-similarity} and \emph{coamenability}, in a way that requires stronger assumptions than those we use here to prove vanishing of second bounded cohomology (see \cite[Theorem 4, Corollary 5]{lamplighters}). In particular, the results from \cite{lamplighters} are not enough to establish the vanishing of bounded cohomology for \emph{arbitrary} groups of piecewise linear or piecewise projective homeomorphisms, among whom the self similar groups are a very sparse subclass. Thus, to reach the degree of generality that we present in this paper, our Theorem \ref{thm:main} is still the best approach. Moreover, Monod's method do not appear to have relevance to our solution of Question \ref{qNavas}, or to the solution of the homological von Neumann--Day problems, since the groups $G_\rho$ are far from being self similar, and in both of the other situations Theorem \ref{thm:main} is applied to \emph{arbitrary} groups of piecewise linear, respectively piecewise projective, homeomorphisms.
\end{remark}

\textbf{Organization:} We start by laying out the necessary background on groups acting on $1$-manifolds and bounded cohomology in Section \ref{s:preli}. In Section \ref{s:criterion} we prove Theorem \ref{thm:main} and its first applications, including Theorem \ref{thm:PLPP}. In Section \ref{s:grho} we introduce the groups $G_\rho$ and analyze their structure, with the aim to prove Theorem \ref{theorem:mainNavas}. We move to circle actions in Section \ref{s:circle}, proving Theorem \ref{thm:T}. Section \ref{s:scl} is dedicated to the golden ratio Thompson group, and Theorem \ref{thm:scl}. Finally, in Section \ref{s:simvol} we review some background on simplicial volume and prove Theorem \ref{thm:sv}. \\

\textbf{Acknowledgements:} The first author was supported by an ETH Z\"urich Doc.Mobility Fellowship. The second author was supported by a START-preis grant of the Austrian science fund. The authors would like to thank Elena Bogliolo, Jonathan Bowden, Matt Brin, Danny Calegari, Roberto Frigerio, Nicolaus Heuer, Alessandra Iozzi, Clara L\"oh, Nicolas Monod, Justin Moore, Marco Moraschini, Andr{\'e}s Navas and Matt Zaremsky for several useful discussion and comments. They are also indebted to anonymous referees for useful comments on earlier versions of this paper (and on the paper \cite{our}, which was later merged to the present paper), in particular for the arguments in Subsection \ref{ss:nico}.

\section{Preliminaries}
\label{s:preli}

Group actions will always be on the right. Accordingly, we will use the conventions $[g, h] = g^{-1} h^{-1} g h$ for commutators, and $g^h := h^{-1}gh$ for conjugacy. All groups considered in this paper will be discrete. Since the applications to simplicial volume are of a different, more geometric-topological flavour, we delay the relevant preliminaries to the corresponding section.

\subsection{Dynamics of group actions on $1$-manifolds}
\label{ss:dynamics}

Given $g\in \Homeo^+(M)$ for a given $1$-manifold $M$, we define the \emph{support of $g$} as: $$Supp(g) := \{x\in M\mid x\cdot g\neq x\}.$$
More generally, given a subgroup $G\leq \Homeo^+(M)$, we define 
$$Supp(G) := \{x\in M\mid \exists g\in G\text{ such that } x\cdot g\neq x\}.$$
We identify $\mathbf{S}^1$ with $\mathbf{R}/\mathbf{Z}$.
Recall that $\Homeo^+(\mathbf{R})$ and $\Homeo^+(\mathbf{R/Z})$ are the groups of orientation-preserving homeomorphisms of the real line and the circle. The latter group action admits a so-called \emph{lift} to the real line, which is the group $\overline{\Homeo^+(\mathbf{R}/\mathbf{Z})}\leq \Homeo^+(\mathbf{R})$ which equals the full centralizer of the group of integer translations $\mathbf{Z}=\{x\mapsto x+n\mid n\in \mathbf{Z}\}$ inside $\Homeo^+(\mathbf{R})$. There is a short exact sequence $$1\to \mathbf{Z}\to \overline{\Homeo^+(\mathbf{R}/\mathbf{Z})} \to \Homeo^+(\mathbf{R}/\mathbf{Z}) \to 1.$$
In this paper, we fix the notation for the map $$\eta: \overline{\Homeo^+(\mathbf{R}/\mathbf{Z})}\to \Homeo^+(\mathbf{R}/\mathbf{Z}).$$

Given a group $G\leq \Homeo^+(\mathbf{R}/\mathbf{Z})$, we define the \emph{total lift}
of $G$ to be the group $\eta^{-1}(G)$, which is a central extension of $G$ by $\mathbf{Z}$. Throughout this paper, we will denote the total lift of $G$ by $\overline{G}$.
More generally, we may define a \emph{lift} of $G$ as a group action $H\leq \overline{\Homeo^+(\mathbf{R}/\mathbf{Z})}$ satisfying that $\eta(H)=G$. Such a lift may not be unique.
On the other hand, the total lift of $G$ is the group generated by all possible such lifts. \\

A group action $G\leq \Homeo^+(\mathbf{R})$ is said to be \emph{boundedly supported} if for each $f\in G$ there is a bounded interval $I\subset \mathbf{R}$ such that $Supp(f) \subset I$.
A group action $G\leq \Homeo^+(M)$ for a connected $1$-manifold $M$ is said to be \emph{proximal}, if for each pair of nonempty open sets $I,J\subset M, J\neq M$ satisfying that the closure of $J$ is compact, there is an element $f\in G$ such that $J\cdot f\subset I$.
%This notion is sometimes referred to as \emph{strongly proximal}. 
In the case of minimal actions on the circle, it is equivalent to the \emph{contracting} property: for every arc $I$ there exists a sequence $(f_n)_{n \geq 1} \subset G$ such that the diameter of $I \cdot f_n$ tends to $0$ as $n \to \infty$ \cite[p. 12]{burger}.

% Recall that a group $G$ is \emph{$n$-uniformly simple} for a fixed $n\in \mathbf{N}$, if for each pair of elements $f,g\in G\setminus \{id\}$, $f$ can be expressed as a product of at most $n$ conjugates of the elements $g,g^{-1}$.
% $G$ is said to be \emph{$n$-uniformly perfect} if each element $f\in G\setminus \{id\}$
% can be expressed as a product of at most $n$ commutators of elements in $G$. 
% A group is uniformly simple (or uniformly perfect) if there exists an $n\in \mathbf{N}$ such that the group is $n$-uniformly simple (respectively, $n$-uniformly perfect).
% It is easy to see that uniformly simple groups are uniformly perfect.
% The following is a special case of Theorem $1.1$ in \cite{prox}.

% \begin{theorem}[Gal--Gismatullin \cite{prox}]
% \label{thm:uniformly simple}
% Let $G\leq \Homeo^+(\mathbf{R})$ be a boundedly supported and promixal group action.
% Then $G'$ is $6$-uniformly simple.
% \end{theorem}

\subsection{Stable commutator length}

%The following is a direct consequence of the definition of stable commutator length as stated in the introduction.
% \begin{lemma}\label{lem:upscl}
% Let $G$ be a uniformly perfect group. Then $\scl$ vanishes everywhere on $G$.
% \end{lemma}

For a comprehensive survey on the topic, we refer the reader to \cite[Chapter 2]{calegari}. The way we will compute $\scl$ is via quasimorphisms.

\begin{definition}
Let $\phi : G \to \mathbf{R}$. Its \emph{defect} is
$$D(\phi) := \sup\limits_{g, h \in G} \left|\phi(g) + \phi(h) - \phi(gh)\right|.$$
If the defect is finite, we say that $\phi$ is a \emph{quasimorphism}. If moreover $\phi(g^k) = k \phi(g)$ for every $g \in G$ and every $k \in \mathbf{Z}$, we say that $\phi$ is \emph{homogeneous}.
\end{definition}

The following elementary property of homogeneous quasimorphisms will be useful:

\begin{lemma}[{\cite[2.2.3]{calegari}}]
\label{lem:conjugacyqm}

Homogeneous quasimorphisms are conjugacy-invariant.
\end{lemma}

Every quasimorphism is at a bounded distance from a unique homogeneous quasimorphism \cite[Lemma 2.21]{calegari}. The following example is the fundamental one for our purposes:

\begin{example}[Poincar\'e, see {\cite[11.1]{ds}}]
\label{ex:rot}

Let $g \in \overline{\Homeo^+(\mathbf{R}/\mathbf{Z})}$ be an orientation-preserving homeomorphism of the line that commutes with integer translations. Define the \emph{rotation number} of $g$ to be
$$\rot(g) := \lim\limits_{n \to \infty} \frac{0 \cdot g^n}{n}.$$
Then $\rot : \overline{\Homeo^+(\mathbf{R}/\mathbf{Z})} \to \mathbf{R}$ is a homogeneous quasimorphism of defect $1$ \cite[Proposition 2.92]{calegari}.
It follows that for every subgroup $G \leq \overline{\Homeo^+(\mathbf{R}/\mathbf{Z})}$, the restriction of $\rot$ to $G$ is a homogeneous quasimorphism of defect at most $1$. Note however that the defect of $\rot|_G$ need not be equal to $1$: for example $\rot|_{\mathbf{Z}}$ is a homomorphism, so its defect is $0$.
\end{example}

The connection between $\scl$ and quasimorphisms is provided by the following fundamental result (see \cite[Theorem 2.70]{calegari}):

\begin{theorem}[Bavard Duality \cite{bavard}]
\label{thm:bavard}

Let $G$ be a group and let $g \in G'$. Then:
$$\scl(g) = \frac{1}{2} \sup \frac{|\phi(g)|}{D(\phi)};$$
where the supremum runs over all homogeneous quasimorphisms of positive defect.
\end{theorem}

\begin{corollary}[Bavard \cite{bavard}]
\label{cor:bavard}
Let $G$ be a group. Then $\scl$ vanishes everywhere on $G'$ if and only if every homogeneous quasimorphism on $G$ is a homomorphism.
\end{corollary}

In particular, if $G$ is such that the space of homogeneous quasimorphisms is one-dimensional and spanned by $\phi$, then the computation of $\scl$ reduces to the evaluation of $\phi$ and its defect. This is the approach taken in \cite{zhuang} to produce examples of groups with transcendental $\scl$, where the $\phi$ in question is the rotation quasimorphism.

\subsection{(Bounded) cohomology and central extensions}
\label{ss:bc}

We will work with cohomology and bounded cohomology with trivial real coefficients, and use the definition in terms of the bar resolution. We refer the reader to \cite{brown_book} and \cite{BC} for a general and complete treatment of ordinary and bounded cohomology (of discrete groups), respectively. \\

For every $n \geq 0$, denote by $\CC^n(G)$ the set of real-valued functions on $G^n$. By convention, $G^0$ is a single point, so $\CC^0(G) \cong \mathbf{R}$ consists only of constant functions. We define differential operators $\delta^\bullet : \CC^\bullet(G) \to \CC^{\bullet+1}(G)$ as follows:
$\delta^0 = 0$
and for $n \geq 1$:
\begin{align*}
    \delta^n(f)(g_1, \ldots, g_{n+1}) &= f (g_2, \ldots, g_{n+1}) \\
    &+ \sum\limits_{i = 1}^n (-1)^i f(g_1, \ldots, g_i g_{i+1}, \ldots, g_{n+1}) \\
    &+ (-1)^{n+1} f(g_1, \ldots, g_n).
\end{align*}

One can check that $\delta^{\bullet+1} \delta^{\bullet} = 0$, so $(\CC^\bullet(G), \delta^\bullet)$ is a cochain complex. We denote by $\ZZ^\bullet(G) := \ker(\delta^\bullet)$ the \emph{cocycles}, and by $\BB^\bullet(G) := \im(\delta^{\bullet-1})$ the \emph{coboundaries}. The quotient $\HH^\bullet(G) := \ZZ^\bullet(G) / \BB^\bullet(G)$ is the \emph{cohomology of $G$ with trivial real coefficients}. We will also call this \emph{ordinary cohomology} to make a clear distinction from the bounded one, which we proceed to define. \\

Restricting to functions $f :G^\bullet \to \mathbf{R}$ which are bounded, meaning that their supremum $\| f \|_\infty$ is finite, leads to a subcomplex $(\CC^\bullet_b(G), \delta^\bullet)$. We denote by $\ZZ^\bullet_b(G)$ the \emph{bounded cocycles}, and by $\BB^\bullet_b(G)$ the \emph{bounded coboundaries}. The vector space $\HH^\bullet_b(G) := \ZZ^\bullet_b(G) / \BB^\bullet_b(G)$ is the \emph{bounded cohomology of $G$} with trivial real coefficients.

The inclusion of the bounded cochain complex into the ordinary one induces a linear map at the level of cohomology, called the \emph{comparison map}:
$$c^\bullet : \HH^\bullet_b(G) \to \HH^\bullet(G).$$
This map is in general neither injective nor surjective. The kernel defines a sequence of subspaces $\EH^\bullet_b(G) \subset \HH^\bullet_b(G)$, called \emph{exact bounded cohomology}.

Cohomology is a contravariant functor: given a cocycle $z \in \ZZ^n(H)$ and a group homomorphism $\varphi : G \to H$, the \emph{pullback} $\varphi^* z := z \circ \varphi$ is a cocycle in $\ZZ^n(G)$. This also induces a pullback map in cohomology $\varphi^* : \HH^n(H) \to \HH^n(G)$, which is of course linear. The same holds for bounded cohomology.

A peculiarity of bounded cohomology classes is that they admit a further invariant, which is their \emph{Gromov seminorm} $\| \cdot \|$. This is simply the quotient seminorm induced by the supremum norm on the quotient $\HH^\bullet_b(G) := \ZZ^\bullet_b(G) / \BB^\bullet_b(G)$. In general it need not be a norm, i.e. there could be non-zero classes with zero norm. However this is not the case in degree $2$ \cite{MM}: here for every group the Gromov seminorm is always a norm.

One can similarly define cohomology and bounded cohomology with \emph{trivial integral coefficients} $\HH^n_b(G; \mathbf{Z})$: here $\mathbf{R}$ is replaced by $\mathbf{Z}$ with its Euclidean norm, equipped with the trivial $G$-action. This will only make a brief appearance in Subsection \ref{ss:eulerclass} and in Corollary \ref{cor:integral}. \\

For trivial real coefficients, the kernel of the second comparison map admits a description in terms of quasimorphisms:

\begin{theorem}
\label{thm:hbqm}

Let $Q(G)$ denote the space of homogeneous quasimorphisms on $G$, and $\ZZ^1(G)$ the subspace of homomorphisms. Then the sequence
$$0 \to \ZZ^1(G) \hookrightarrow Q(G) \xrightarrow{[\delta^1 \cdot]} \HH^2_b(G) \xrightarrow{c^2} \HH^2(G)$$
is exact. In particular, $c^2$ is injective if and only if every homogeneous quasimorphism on $G$ is a homomorphism.
\end{theorem}

Combining this with Corollary \ref{cor:bavard}, we deduce:

\begin{corollary}
Let $G$ be a group such that $\HH^2_b(G) = 0$. Then $\scl$ vanishes everywhere on $G'$. %This holds in particular if $G$ is amenable.
\end{corollary}

The main example of groups with vanishing bounded cohomology is the following:

\begin{theorem}[Johnson, see {\cite[Chapter 3]{BC}}]
\label{thm:amenable}

Let $G$ be an amenable group. Then $\HH^n_b(G) = 0$ for all $n \geq 1$.
\end{theorem}

The converse does not hold: there exist groups whose bounded cohomology vanishes in all positive degrees, but contain free subgroups. The first such example is due to Matsumoto and Morita \cite{MM}: the group of compactly supported homeomorphisms of $\mathbf{R}^n$.
%Finitely generated and finitely presented examples have recently been constructed \cite{bcfg}. \\

Bounded cohomology also behaves well with respect to amenable extensions:

\begin{theorem}[Gromov, see {\cite[Chapter 4]{BC}}]
\label{thm:mapping}

Let $n \geq 0$, and let
$$1 \to H \to G \to K \to 1$$
be a group extension such that $H$ is amenable. Then the quotient $G \to K$ induces an isomorphism in bounded cohomology $\HH^n_b(K) \to \HH^n_b(G)$.
\end{theorem}

\begin{theorem}[{Monod \cite[8.6]{monod}, see also \cite{coamenable}}]
\label{thm:amenableext}

Let $n \geq 0$, and let
$$1 \to H \to G \to K \to 1$$
be a group extension such that $K$ is amenable. Then the inclusion $H \to G$ induces an injection in bounded cohomology $\HH^n_b(G) \to \HH^n_b(H)$.
\end{theorem}

%\subsection{Central extensions}

In degree $2$, cohomology is strongly related to central extensions, and this relation can be exploited to study bounded cohomology as well. We refer the reader to \cite[Chapter 4, Section 3]{brown_book} and \cite[Chapter 2]{BC} for detailed accounts and proofs.

Since we will mostly be working with trivial real coefficients, throughout this paragraph all central extensions will be of the form:
$$1 \to \mathbf{R} \to E \to G \to 1;$$
where $R := \im(\mathbf{R} \to E)$ is contained in the center of $E$ (central extensions with integral kernel will appear only briefly to define the Euler class in the next subsection). We will refer to both the short exact sequence above and the group $E$ as a central extension. Such an extension \emph{splits} if there exists a homomorphic section $\sigma : G \to E$, i.e. if it is a direct product. Two central extensions $E, E'$ are \emph{equivalent} if there exists a homomorphism $f : E \to E'$ such that the following diagram commutes:
\[\begin{tikzcd}
	&& E \\
	1 & \mathbf{R} && G & 1 \\
	&& {E'}
	\arrow[from=2-1, to=2-2]
	\arrow[from=2-2, to=1-3]
	\arrow[from=1-3, to=2-4]
	\arrow[from=2-4, to=2-5]
	\arrow[from=2-2, to=3-3]
	\arrow[from=3-3, to=2-4]
	\arrow["f"', from=1-3, to=3-3]
\end{tikzcd}\]
Note that by the $5$-lemma $f$ is automatically an isomorphism. A central extension is split if and only if it is equivalent to the external direct product $\mathbf{R} \times G$.

We consider pairs of the form $(E,\sigma)$ where $E$ is a central extension and $\sigma : G \to E$ is a set-theoretic section, which is moreover \emph{normalized}, i.e. it satisfies that $\sigma(\id_G) = \id_E$.
To every such pair one can associate a $2$-cocycle as follows. Observe that the map $$G^2 \to E : (f, g) \mapsto \sigma(f)\sigma(g)\sigma(fg)^{-1}$$ takes values in $R$. Therefore this can be viewed as a map $\omega : G^2
\to \mathbf{R}$ which can be verified to be a cocycle. This is moreover \emph{normalized}, that is it satisfies $\omega(f, \id_G) = \omega(\id_G, f) = 0$ for all $f \in G$.

Conversely, let $\omega$ be a normalized cocycle. Define a group $E$ as follows: as a set, $E = \mathbf{R} \times G$. The group law is defined by the following formula:
$$(\lambda, f) \cdot (\mu, g) := (\lambda + \mu + \omega(f, g), fg).$$
Then $E$ is a group, and the set-theoretic inclusions $\mathbf{R} \to E$ and projection $E \to G$ make $E$ into a central extension. Moreover, the set-theoretic inclusion $\sigma : G \to E$ is a normalized section satisfying $\omega(f, g) = \sigma(f)\sigma(g)\sigma(fg)^{-1}$.

This correspondence between normalized cocycles and central extensions with preferred normalized sections, descends to the level of cohomology:

\begin{theorem}[{\cite[Section 4, Chapter 3]{brown_book}}]
\label{thm:ext}

There is a bijective correspondence between normalized $2$-cocycles $\omega \in \ZZ^2(G)$ and central extensions $E$ with a normalized section $\sigma : G \to E$. This induces a bijective correspondence between cohomology classes $\alpha \in \HH^2(G)$ and equivalence classes $[E]$ of central extensions.
\end{theorem}

Recall that if $E$ is a central extension and $p : E \to G$ is the quotient map, then there is an induced map in cohomology $p^* : \HH^2(G) \to \HH^2(E)$. The effect of this map is to annihilate the class represented by $E$:

\begin{lemma}[{\cite[Lemma 2.3]{BC}}]
\label{lem:pullback0}

Let $E$ be a central extension of $G$ by $\mathbf{R}$ and let $\alpha \in \HH^2(G)$ be the corresponding cohomology class. Then $p^* \alpha = 0 \in \HH^2(E)$.
\end{lemma}

The relation of second cohomology with central extensions makes second bounded cohomology more well-behaved than other degrees. As an instance of this principle, we have the following fact, which is not known to hold in higher degrees:

\begin{proposition}[{\cite[Corollary 4.16]{bcbinate}}]
\label{prop:dirun}

Let $G$ be a directed union of a family of subgroups $(G_i)_{i \in I}$ such that $\HH^2_b(G_i) = 0$ for all $i \in I$. Then $\HH^2_b(G) = 0$.
\end{proposition}

\subsection{The Euler class of a circle action}
\label{ss:eulerclass}

Bounded cohomology can be used to study actions on the circle: we refer the reader to \cite{ghys} and \cite[Chapter 10]{BC} for a detailed exposition.
In this section we will treat bounded cohomology with both trivial real and trivial integral coefficients. To avoid confusion, we will specify all coefficients in the notation, but will get back to the convention that $\HH^2_b(G) = \HH^2_b(G; \mathbf{R})$ in the next section.

The central extension
$$1\to \mathbf{Z}\to \overline{\Homeo^+(\mathbf{R}/\mathbf{Z})} \to \Homeo^+(\mathbf{R}/\mathbf{Z}) \to 1$$
defines a cohomology class $e \in \HH^2(\Homeo^+(\mathbf{R}/\mathbf{Z}); \mathbf{Z})$. This class admits a bounded representative, which defines a bounded cohomology class $e_b \in \HH^2_b(\Homeo^+(\mathbf{R}/\mathbf{Z}); \mathbf{Z})$. Given a group $G$ and an action $\rho : G \to \Homeo^+(\mathbf{R}/\mathbf{Z})$, the pullback of $e_b$ defines a cohomology class $e_b(\rho) \in \HH^2_b(G; \mathbf{Z})$ called the \emph{(bounded) Euler class of $\rho$}. This class provides a generalization of rotation number theory, and it classifies actions on the circle up to a suitable notion of semiconjugacy.

Since our results concern real bounded cohomology, we will be working instead with the \emph{real Euler class}. This is denoted by $e_b^{\mathbf{R}}(\rho)$ and is defined as the image of $e_b(\rho)$ under the change of coefficients map $\HH^2_b(G; \mathbf{Z}) \to \HH^2_b(G; \mathbf{R})$, which is the map at the level of cohomology induced by the inclusion at the level of cochain complexes $\CC^2_b(G; \mathbf{Z}) \to \CC^2_b(G; \mathbf{R})$. One loses some information when passing to real coefficients, however something can still be said in general:

\begin{theorem}[Burger \cite{burger}]
\label{thm:burger}

Let $\rho$ be a proximal action of $G$ on the circle. Then $\| e_b^{\mathbf{R}}(\rho) \| = 1/2$. If $\rho'$ is another such action, and the two are not conjugate inside $\Homeo^+(\mathbf{R}/\mathbf{Z})$, then $\| e_b^{\mathbf{R}}(\rho) - e_b^{\mathbf{R}}(\rho') \| = 1$.
\end{theorem}

In the statement $\| \cdot \|$ denotes the Gromov seminorm on bounded cohomology (see Subsection \ref{ss:bc}).
Note that in \cite{burger} the statement is given for minimal strongly proximal actions; the definition is however equivalent to our definition of proximal action \cite[p. 12]{burger}.

\begin{corollary}
\label{cor:onedim}

Let $G$ be such that $\HH^2_b(G; \mathbf{R})$ is one-dimensional. Then $G$ admits at most one conjugacy class of proximal actions on the circle.
\end{corollary}

\begin{proof}
There are only two classes in $\HH^2_b(G; \mathbf{R})$ which have norm $1/2$: call them $e$ and $-e$. Suppose that $e$ is the Euler class of a circle action. Then $-e$ is the Euler class of the same action conjugated by an orientation-reversing homeomorphism of the circle \cite[Fin de la d{\'e}monstration du Th{\'e}or{\`e}me B]{ghys}.
\end{proof}

\section{The criterion}
\label{s:criterion}

A special feature of second bounded cohomology is that every class admits a unique canonical representative. This was proven by Bouarich in \cite{bouarich}, as a tool to show that an epimorphism induces an embedding in second bounded cohomology, and will be fundamental for our proof of Theorem \ref{thm:main}.

\begin{definition}
Let $G$ be a group. A bounded $2$-cocycle $\omega \in \ZZ^2_b(G)$ is \emph{homogeneous} if $\omega(g^i, g^j) = 0$ for every $g \in G$ and every $i, j \in \mathbf{Z}$.
\end{definition}

Clearly every homogeneous cocycle is normalized.

\begin{theorem}[Bouarich \cite{bouarich}]
\label{thm:bouarich}
Every second bounded cohomology class admits a unique homogeneous representative.
\end{theorem}

See \cite[Proposition 2.16]{BC} for a detailed proof. This theorem is especially useful in the case in which the class is trivial, since the $0$ cocycle is also homogeneous:

\begin{corollary}
\label{cor:bouarich}

Let $\omega \in \ZZ^2_b(G)$ be a bounded homogeneous $2$-cocycle. Let $E$ be the corresponding central extension, and $\sigma : G \to E$ the normalized section.

Let $H \leq G$ be a group such that $\HH^2_b(H) = 0$. Then $\sigma|_H : H \to E$ is a homomorphism. This holds in particular if $H$ is abelian.
\end{corollary}

\begin{proof}
By Theorem \ref{thm:bouarich} the only bounded homogeneous $2$-cocycle on $H$ is the $0$-cocycle. In particular $\omega|_H \equiv 0$. In other words, $\sigma|_H : H \to E$ is a homomorphism.
\end{proof}

We now demonstrate that for homogeneous $2$-cocycles, the corresponding sections behave well with respect to conjugacy.

\begin{lemma}
\label{lem:conjugacy}
Let $\omega \in \ZZ^2_b(G)$ be a bounded homogeneous $2$-cocycle. Let $E$ be the corresponding central extension, and $\sigma : G \to E$ the normalized section. Then for each pair $h,k\in G$, it holds that $\sigma(h^k) = \sigma(h)^{\sigma(k)}$.
\end{lemma}

\begin{proof}
Let $p : E \to G$ denote the quotient map. Then $p^* \omega$ is a bounded homogeneous $2$-cocycle on $E$. Moreover, $p^* \omega$ represents the trivial second cohomology class, by Lemma \ref{lem:pullback0}. Therefore there exists a map $\phi : E \to \mathbf{R}$ such that $\delta^1 \phi = p^*\omega$. Since $p^*\omega$ is bounded, $\phi$ is a quasimorphism; and since $p^* \omega$ is homogeneous, $\phi$ is a homogeneous quasimorphism. In particular $\phi$ is conjugacy-invariant, by Lemma \ref{lem:conjugacyqm}.

It follows from Corollary \ref{cor:bouarich} that $\sigma$ is a homomorphism when restricted to a cyclic subgroup. Let $g=\sigma(h), f=\sigma(k) \in E$. We compute
\begin{align*}
    \sigma(h)^{\sigma(k)}\sigma(h^k)^{-1} &= \sigma(k)^{-1}\sigma(h)\sigma(k)\sigma(k^{-1}hk)^{-1} \\
    &= \sigma(p(f)^{-1})\sigma(p(g)) \sigma(p(f))\sigma(p(f)^{-1}p(g)p(f))^{-1} \\
    &= \sigma(p(f)^{-1})\sigma(p(g))\bigg(\sigma(p(f)^{-1}p(g))^{-1} \cdot \sigma(p(f)^{-1}p(g))\bigg) \sigma(p(f))\sigma(p(f)^{-1}p(g)p(f))^{-1} \\
    &= p^*\omega(f^{-1},g)\cdot p^*\omega(f^{-1}g,f) \\
    &= (\phi(f^{-1}) + \phi(g) - \phi(f^{-1}g)) + (\phi(f^{-1}g) + \phi(f) - \phi(f^{-1}gf)).
\end{align*}
Now using that $\phi$ is homogeneous, and thus conjugacy-invariant, the last expression vanishes.
\end{proof}

Now we are ready to prove Theorem \ref{thm:main}.

\begin{proof}[Proof of Theorem \ref{thm:main}]
%Let $G$ be a group with commuting conjugates. First we claim that for every finitely generated subgroup $H \leq G$ and every $k \geq 1$, there exist elements $c_1, c_2, \ldots, c_k \in G$ such that $[H^{c_i}, H^{c_j}] = id$ for every $i \neq j$. We proceed by induction. Suppose that for $k \geq 2$, there are elements $c_1, c_2, \ldots, c_k\in H$ such that $[H^{c_i}, H^{c_j}] = id$ for every $i \neq j$. The group $K$ generated by $\{H^{c_1},\ldots,H^{c_k}\}$ is finitely generated. By our hypothesis, there exists $c_{k+1} \in G$ such that $[K, K^{c_{k+1}}] = id$. In particular $[H^{c_i}, H^{c_{k+1}}] = id$ for every $i = 1, \ldots, k$. 
%Using this, we proceed to prove the statement.

Suppose by contradiction that there exists a class $\alpha \in \HH^2_b(G)$ which is non-trivial. Let $\omega : G^2 \to \mathbf{R}$ be the unique homogeneous representative given by Theorem \ref{thm:bouarich}. By Theorem \ref{thm:ext}, this corresponds to a central extension
$$1 \to \mathbf{R} \to E \to G \to 1,$$
endowed with a section $\sigma : G \to E$ such that $\omega(f, g) =  \sigma(f)\sigma(g)\sigma(fg)^{-1}$ for every $f, g \in G$, and $\sigma(\id_G) = \id_E$. Moreover, since $\omega$ is non-trivial, $\sigma$ cannot be a homomorphism. Thus there must be a relation $f_1 f_2 f_3 = \id_G$ in $G$, for which
$$\sigma(f_1) \sigma(f_2) \sigma(f_3) = \lambda \in \mathbf{R} \leq E, \qquad \lambda \neq 0.$$
Note that $f_3 = (f_1 f_2)^{-1}$, so $\sigma(f_3) = \sigma(f_1 f_2)^{-1}$. Therefore $|\lambda| \leq \| \omega \|_\infty$ (recall that $\omega$ is a bounded cocycle), and we choose $f_1, f_2, f_3$ so that $|\lambda| > \| \omega \|_\infty / 2$.

Let $H \leq G$ be the group generated by $f_1, f_2, f_3$. The commuting conjugates condition ensures the existence of an element $c \in G$ such that $H$ and $H^c$ commute. By Corollary \ref{cor:bouarich}, this has the following consequences:
$$\sigma(f_{i_1} f_{i_2}^c) = \sigma(f_{i_1}) \sigma(f_{i_2}^c); \qquad [\sigma(f_{i_1}), \sigma(f_{i_2}^c)] = \id_E; \qquad \text{for all } i_j \in \{1, 2, 3\}.$$

Now let $F_i := f_i f_i^c$. Using the commuting relation, we have:
$$F_1 F_2 F_3 = (f_1 f_1^c) \cdot (f_2 f_2^c) \cdot (f_3 f_3^c) = (f_1 f_2 f_3) \cdot (f_1 f_2 f_3)^c = \id_G.$$
Therefore $\sigma(F_1) \sigma(F_2) \sigma(F_3) \in \mathbf{R}$. The following claim computes its value exactly:

\begin{claim}
It holds: $\sigma(F_1) \sigma(F_2) \sigma(F_3) = 2 \lambda \in \mathbf{R}$.
\end{claim}

But this implies that $\| \omega \|_\infty \geq |2 \lambda|$, which contradicts our assumption that $|\lambda| > \| \omega \|_\infty / 2$, and concludes the proof. \\

%Let us first see how the claim implies the theorem. Since $\lambda \neq 0$, and the above argument is valid for every $k \geq 1$, we can choose $k$ to be large enough that $|k \lambda| > (n-1) \|\omega\|_\infty$. Then the claim implies that
%$$|\sigma(F_1) \cdots \sigma(F_n) \sigma(F_1 \cdots F_n)^{-1}| = |\sigma(F_1) \cdots \sigma(F_n)| = |k \lambda| > (n-1) \| \omega \|_\infty,$$
%contradicting Lemma \ref{lem:multiproduct}. This concludes the proof. \\
Therefore we are left to prove the claim. Using the commuting relations on the $\sigma(f_i), \sigma(f_i^c)$, we obtain:
\begin{align*}
    \sigma(F_1) \sigma(F_2) \sigma(F_3) &= (\sigma(f_1) \sigma(f_1^c)) \cdot (\sigma(f_2) \sigma(f_2^c)) \cdot (\sigma(f_3) \sigma(f_3^c)) \\
    &= (\sigma(f_1) \sigma(f_2) \sigma(f_3)) \cdot (\sigma(f_1^c) \sigma(f_2^c) \sigma(f_3^c)).
\end{align*}
Using Lemma \ref{lem:conjugacy}, and the fact that $\mathbf{R}$ is central, we obtain
$$\sigma(f_1^c) \sigma(f_2^c) \sigma(f_3^c) = \sigma(f_1)^{\sigma(c)} \sigma(f_2)^{\sigma(c)} \sigma(f_3)^{\sigma(c)} = (\sigma(f_1) \sigma(f_2) \sigma(f_3))^{\sigma(c)} = \lambda.$$
Thus the previous expression equals $2 \lambda$, which is what we wanted to prove.

%\fff{Modified proof: please double-check. I commented out the end of section 2.3 and the beginning of this proof, which are not needed anymore}

\end{proof}

Combining this with Theorem \ref{thm:ext}, we obtain:

\begin{corollary}
\label{cor:main}

Let $G$ be a group extension of the form
$$1 \to H \to G \to K \to 1$$
where $H$ has commuting conjugates and $K$ is amenable. Then $\HH^2_b(G) = 0$.
\end{corollary}

\begin{remark}
An extension of two groups with vanishing second bounded cohomology has the same property \cite[Corollary 4.2.2]{bac}. So Corollary \ref{cor:main} can be phrased for more general extensions as well.
\end{remark}

\subsection{First applications}

We start with a very general example of groups with commuting conjugates. This will serve as a base case for the proof of Theorem \ref{thm:PLPP}.

\begin{proposition}\label{prop:locc}
Let $G$ be a group of boundedly supported homeomorphisms of the real line with no global fixpoint. Then $G$ has commuting conjugates. In particular $\HH^2_b(G) = 0$.
\end{proposition}

\begin{proof}
Let $H \leq G$ be a finitely generated subgroup. There is a bounded, open interval $(a, b)$ in $\mathbf{R}$ on which $H$ is supported. If the orbit $a \cdot G$ had a supremum in $\mathbf{R}$, then this would be a global fixpoint. Therefore, there exists $f \in G$ such that $a \cdot f > b$ and so $(a, b) \cap (a, b) \cdot f = \emptyset$. It follows that $H$ and $H^f$ commute, since they have disjoint support.
\end{proof}

%In case the action of $G$ is moreover proximal, it was already known that the commutator subgroup $G'$ is uniformly perfect, in fact uniformly simple (Theorem \ref{thm:uniformly simple}). Therefore $\EH^2_b(G) = 0$ by Corollary \ref{cor:bavard} and Theorem \ref{thm:hbqm}. This result extends this known fact to the whole second bounded cohomology, with a weaker set of hypotheses. \\

We now show that groups of piecewise linear and projective homomorphisms satisfy the hypotheses of Corollary \ref{cor:main}.

\begin{proof}[Proof of Theorem \ref{thm:PLPP} part 1]
We will prove this for an arbitrary group $G$ of piecewise projective homeomorphisms of the real line.
The proof for the case of groups of piecewise linear homeomorphisms of the unit interval is similar.
We define the sets $$Fix(G)=
%\{x\in \mathbf{R}\mid \forall f\in G, x\cdot f=x\}=
\mathbf{R}\setminus Supp(G)\qquad Tr(G)=%\{x\in Fix(G)\mid \forall 
%\epsilon>0, \exists y\in (x-\epsilon, x+\epsilon) \exists f\in G,  y\cdot f\neq y\}=
\overline{Supp(G)}\setminus Supp(G).$$ 

%The group $H$ of germs of elements in $G$ at each point in $Tr(G)$ is metabelian (hence amenable). This is because in the group of piecewise projective homeomorphisms of the real line the following holds. 

Recall that in the projective action of $\textup{PSL}_2(\mathbf{R})$ on $\mathbf{S}^1=\mathbf{R}\cup \{\infty\}$, the stabilizer of $\infty$ is the affine group. And since the action of $\textup{PSL}_2(\mathbf{R})$ on $\mathbf{S}^1=\mathbf{R}\cup \{\infty\}$ is transitive, the stabilizer of any point $x\in \mathbf{S}^1$ is conjugate to the affine group.
Therefore, for $x\in Tr(G)$, the group of right (or left) germs of $G$ at $x$ is conjugate to a subgroup of the affine group and hence amenable (in fact metabelian).
For the piecewise linear case, one similarly argues that the groups of germs are abelian.
It follows that there is a short exact sequence 
$$1\to H\to G\to K\to 1$$
where $K$ is the group of germs at $Tr(G)$ in $G$ and $H$ is the group consisting of elements
$$H=\{f\in G\mid \forall x\in Tr(G), x\notin \overline{Supp(f)}\}$$ 

By Corollary \ref{cor:main}, to establish that $\HH^2_b(G)=0$, it suffices to show that $H$ has commuting conjugates.
Let $\Gamma$ be a finitely generated subgroup of $H$. By the description of elements of $H$ as above, and since each piecewise projective homeomorphism of the real line has finitely many components of support, there are finitely many open intervals $I_1,\ldots,I_k$ in $(-\infty, \infty)\setminus Fix(G)$ such that:
\begin{enumerate}

\item The intervals $I_1,\ldots,I_k$ are connected components of $Supp(G)$.

\item $\overline{Supp(\Gamma)}\subset \bigcup_{1\leq i\leq k}I_i$ and $Supp(\Gamma)\cap I_i\neq \emptyset$ for each $1\leq i\leq k$.

\end{enumerate}
For each $1\leq i\leq k$, we let $$J_i=Supp(\Gamma)\cap I_i\qquad J=\bigcup_{1\leq i\leq k}J_i=Supp(\Gamma).$$
Our goal is to find an element $\gamma\in H$ such that $J\cdot \gamma\cap J=\emptyset$.
We proceed by induction on $k$, the case $k = 1$ being contained in Proposition \ref{prop:locc}.
Let $x$ be the right endpoint of the interval $J_k$.
Since the action of $H$ on $I_k$ does not admit a global fixed point, there is an element $g\in H$
such that $x<inf(J_k\cdot g)$.
We apply the inductive hypothesis to the intervals 
$$L_1=J_1\cup J_1\cdot g\qquad \ldots \qquad L_{k-1}=J_{k-1}\cup J_{k-1}\cdot g\qquad L=(\bigcup_{1\leq i\leq k-1}L_i).$$
to obtain an element $f\in H$ such that 
$(L\cdot f)\cap L=\emptyset$.
This also implies that 
$(L\cdot f^{-1})\cap L=\emptyset$.
%$$((\bigcup_{1\leq i\leq k-1}L_i)\cdot f^{-1})\cap \bigcup_{1\leq i\leq k-1}L_i=\emptyset$$
Let $f_1\in \{f,f^{-1}\}$ be the element such that $x\cdot f_1\geq x$. 
Then the element $\gamma=gf_1$ satisfies that $J\cdot \gamma\cap J=\emptyset$.
We conclude that $[\Gamma,\Gamma^{\gamma}]=\id$.
\end{proof}

In \cite{LodhaMoore}, the second author and Justin Moore constructed a finitely presented nonamenable group of piecewise projective homeomorphisms of the real line, denoted by $G_0$. In \cite{Lodha}, it was demonstrated that the group $G_0$ is of type $F_{\infty}$. It follows from a direct application of Theorem \ref{thm:PLPP} that
$\HH^2_b(H)=0$ for every subgroup $H \leq G_0$. This gives a negative answer to the homological von Neumann--Day problem posed by Calegari in \cite{scl_pl}, and mentioned in the introduction.

\begin{proof}[Proof of Theorem \ref{thm:PLPP} part 2]

It follows immediately from the definitions that a group $G$ that is either a chain group, or a group that admits a coherent action on the real line, is an extension $1\to H\to G\to K\to 1$, where:
\begin{enumerate}
\item $H$ is conjugate to a group of boundedly supported homeomorphisms of the real line with no global fixpoint. 
\item $K$ is amenable. (In the case of chain groups, $K=\mathbf{Z}^2$, and in the case of coherent actions it is solvable.)
\end{enumerate}
It follows from an application of Proposition \ref{prop:locc} and Corollary \ref{cor:main} that $\HH^2_b(G)=0$. 
\end{proof}

\section{The groups $G_\rho$}
\label{s:grho}

In this section we introduce the groups $G_\rho$, and prove that they have vanishing second bounded cohomology (Theorem \ref{theorem:mainNavas}), answering Navas's Question \ref{qNavas}.

\subsection{Definitions}

We start by recalling some features of Thompson's group $F$ and a subgroup $H$ that is relevant for the definition of $G_\rho$. We refer the reader to \cite{CFP, Belk} for comprehensive surveys.
We denote by $\textup{PL}^+([0,1])$ the group of orientation-preserving piecewise linear homeomorphisms of $[0,1]$.

\begin{definition}\label{F}
The group $F\leq \textup{PL}^+([0,1])$ consists of homeomorphisms that satisfy the following:
\begin{enumerate}
\item  The breakpoints lie in $\mathbf{Z}[\frac{1}{2}]$;
\item The derivatives, whenever they exist, are integer powers of $2$.
\end{enumerate}
\end{definition}
Here \emph{breakpoints} (or singularity points) are points where the derivative does not exist.
It is well known that $F$ is finitely presented and that $F'$ is simple and consists of precisely the set of elements $g\in F$ such that $\overline{Supp(g)}\subset (0,1)$.
The following subgroup of $F$ will play a key role in the definition of $G_{\rho}$.
\begin{definition}\label{H}
$H$ is the subgroup of elements of $F$ whose slopes at $0,1$ coincide.
\end{definition}

\begin{lemma}[{\cite[Lemma 2.4]{grho}}]
\label{3gen}
$H$ is $3$-generated. $H'$ is simple and consists of precisely the set of elements of $H$ (or $F$) that are compactly supported in $(0,1)$.
In particular, $H'=F'$.
\end{lemma} 

\begin{definition}\label{Hgenset}
We fix a $3$-element generating set $\{\nu_1,\nu_2,\nu_3\}$ for $H$. (The choice of this set is arbitrary.)
\end{definition}

The next is the key notion in the definition of $G_\rho$.

\begin{definition}\label{labelling}
Consider the set $\frac{1}{2}\mathbf{Z}=\{\frac{1}{2}k\mid k\in \mathbf{Z}\}$.
A \emph{labelling} is a map $$\rho:\frac{1}{2}\mathbf{Z}\to \{a,b,a^{-1},b^{-1}\}$$
which satisfies:
\begin{enumerate}
\item $\rho(k)\in\{a,a^{-1}\}$ for each $k\in \mathbf{Z}$.
\item $\rho(k)\in \{b,b^{-1}\}$ for each $k\in \frac{1}{2}\mathbf{Z}\setminus \mathbf{Z}$.
\end{enumerate}
\end{definition}

It is convenient to view $\rho(\frac{1}{2}\mathbf{Z})$ as a bi-infinite word with respect to the usual ordering of the half integers.
A subset $X\subseteq \frac{1}{2}\mathbf{Z}$ is said to be a \emph{block} if it is of the form 
$$\big\{ k,k+\frac{1}{2}, \ldots ,k+\frac{1}{2}n \big\}$$
for some $k\in \frac{1}{2}\mathbf{Z}, n\in \mathbf{N}$.
Each block is endowed with the ordering inherited from $\frac{1}{2}\mathbf{Z}$.
The set of blocks of $\frac{1}{2}\mathbf{Z}$ is denoted as $\mathbf{B}$.
To each labelling $\rho$ and each block $X=\{k, k+\frac{1}{2}, \ldots, k+\frac{1}{2}n\}$, we assign a formal word 
$$W_{\rho}(X) = \rho \big( k \big) \rho \big( k+\frac{1}{2} \big) \ldots \rho \big( k+\frac{1}{2}n \big)$$
which is a word in the letters $\{a,b,a^{-1},b^{-1}\}$.
Such a formal word is called a \emph{subword} of the labelling.
Given a word $w_1 \ldots w_n$ in the letters $\{a,b,a^{-1},b^{-1}\}$, the \emph{formal inverse} of the word is 
$w_n^{-1} \ldots w_1^{-1}$. The formal inverse of $W_{\rho}(X)$ is denoted by $W_{\rho}^{-1}(X)$.

\begin{definition}\label{qplabelling}
A labelling $\rho$ is said to be \emph{quasi-periodic} if the following holds:
\begin{enumerate}
\item For each block $X\in \mathbf{B}$, there is an $n\in \mathbf{N}$ such that whenever 
$Y\in \mathbf{B}$ is a block of size at least $n$, then $W_{\rho}(X)$ is a subword of $W_{\rho}(Y)$.
\item For each block $X\in \mathbf{B}$, there is a block $Y\in \mathbf{B}$ such that $W_{\rho}(Y)=W_{\rho}^{-1}(X)$.
\end{enumerate}
\end{definition}

Note that by \emph{subword} in the above we mean a string of consecutive letters in the word.
An explicit construction of quasi-periodic labellings was provided in \cite[Lemma 3.1]{grho}.
The following is a direct consequence of Item $2$ in Definition \ref{qplabelling}:

\begin{lemma}\label{lem:notperiodic}
A quasi-periodic labelling is not periodic.
\end{lemma}

We shall need the following notation.
Let $I,J$ be nonempty compact intervals of equal length.
We denote by $T_{J,I}:J\to I$ the unique orientation-preserving isometry,
and by $T_{J,I}^{or}:J\to I$ the unique orientation-reversing isometry.
Given two isometric closed intervals $I,J\subset \mathbf{R}$, and homeomorphisms $f\in \textup{Homeo}^+(I),g\in \textup{Homeo}^+(J)$, we say that:

\begin{enumerate}
\item $f\cong_T g$ if $f=T_{I,J}\circ f\circ T_{J,I}.$
\item $f\cong_{T^{or}} g$ if $f=T_{I,J}^{or}\circ f\circ T_{J,I}^{or}.$
\end{enumerate}

Let $H=\langle \nu_1,\nu_2,\nu_3\rangle<\textup{Homeo}^+([0,1])$ be the group from Definition \ref{H}. We define the homeomorphisms 
$$\zeta_1,\zeta_2,\zeta_3,\chi_1,\chi_2,\chi_3:\mathbf{R}\to \mathbf{R}$$ 
as follows: for each $i\in \{1,2,3\}$ and $n\in \mathbf{Z}$,
$$\zeta_i\restriction [n,n+1]\cong_{T}\nu_i \qquad \text{ if } \rho \big( n+\frac{1}{2} \big) = b,$$
$$\zeta_i\restriction [n,n+1]\cong_{T^{or}}\nu_i \qquad \text{ if } \rho \big( n+\frac{1}{2} \big) = b^{-1},$$
$$\chi_i\restriction \big[ n-\frac{1}{2},n+\frac{1}{2} \big] \cong_T \nu_i \qquad \text{ if }\rho(n)=a,$$
$$\chi_i\restriction \big[ n-\frac{1}{2},n+\frac{1}{2} \big] \cong_{T^{or}} \nu_i\qquad \text{ if }\rho(n)=a^{-1}.$$

\begin{definition}\label{SimpleGroup}
To each labelling $\rho$, we 
associate the group
$$G_{\rho}:=\langle \zeta_1^{\pm 1},\zeta_2^{\pm 1},\zeta_3^{\pm 1},\chi_1^{\pm 1},\chi_2^{\pm 1},\chi_3^{\pm 1}\rangle<\textup{Homeo}^+(\mathbf{R}).$$
% We denote the above generating set of $G_{\rho}$ as 
% $$\mathbf{S}_{\rho}:=\{\zeta_1^{\pm 1},\zeta_2^{\pm 1},\zeta_3^{\pm 1},\chi_1^{\pm 1},\chi_2^{\pm 1},\chi_3^{\pm 1}\}.$$
\end{definition}

We also define subgroup 
$$\mathcal{K}:=\langle \zeta_1^{\pm 1},\zeta_2^{\pm 1},\zeta_3^{\pm 1}\rangle$$
of $G_{\rho}$ which is naturally isomorphic to $H$, via the isomorphism
$$\lambda : H \to \mathcal{K} : \nu_i \mapsto \zeta_i , i = 1, 2, 3.$$
Note that the definition of $\mathcal{K}$ requires us to fix a labelling $\rho$ but we denote them as such for simplicity of notation.

\subsection{Structure of $G_\rho$}

The group $G_{\rho}$ is defined for every labelling $\rho$. In the special case in which $\rho$ is quasi-periodic, it is moreover simple, and has useful additional properties (see Subsection \ref{ss:dynamics} for the definitions):

\begin{theorem}[{\cite[Theorem 1.3]{grho}, \cite[Lemmas 5.1, 5.3]{grho}, \cite[Corollary 0.3]{uniformlyperfect}}]
\label{thm:properties}

Let $\rho$ be a quasi-periodic labelling.
Then the group $G_{\rho}$ is simple. Moreover:
\begin{enumerate}
    \item Every element in $G_\rho$ fixes a point in $\mathbf{R}$.
    \item The action of $G_\rho$ on $\mathbf{R}$ is minimal.
    \item The action of $G_\rho$ on $\mathbf{R}$ is proximal.
\end{enumerate}
\end{theorem}

We fix a quasi-periodic labelling for the rest of this section. Our next goal is to recall the notion of \emph{special elements} and the global characterization of $G_\rho$ from \cite{uniformlyperfect}. \\

For each $n\in \mathbf{Z}$, we denote by $\iota_n$ the unique orientation-reversing isometry $\iota_n:[n,n+1)\to (n,n+1]$. We extend this to a map $\iota:\mathbf{R}\to \mathbf{R}$ as $x\cdot \iota=x\cdot \iota_n$ where $n\in \mathbf{Z}$ is so that $x\in [n,n+1)$. (Note that $\iota_n=T_{[n,n+1),(n,n+1]}^{or}$ as defined before, however this more specialised notation simplifies what appears below.)

Given an $x\in \mathbf{R}$ and $k\in \mathbf{N}$, we define a word $\mathcal{W}(x,k)$ as follows. 
Let $y\in \frac{1}{2}\mathbf{Z}\setminus \mathbf{Z}$ such that $x\in [y-\frac{1}{2},y+\frac{1}{2})$.
Then we define 
$$\mathcal{W}(x,k) = \rho ( y-\frac{1}{2}k ) \rho \big( y-\frac{1}{2}(k-1) \big) \ldots \rho \big( y \big) \ldots 
\rho \big( y+\frac{1}{2}(k-1) \big) \rho \big( y+\frac{1}{2} k \big).$$
We denote by $\mathcal{W}^{-1}(x,k)$ the formal inverse of the word $\mathcal{W}(x,k)$.
Given a compact interval $J\subset \mathbf{R}$ with endpoints in $\frac{1}{2}\mathbf{Z}$ and 
$n \in \mathbf{N}$, we define a word $\mathcal{W}(J,k)$ as follows. Let
$$y_1=inf(J)+\frac{1}{2}, \qquad y_2=sup(J)-\frac{1}{2}.$$
Then we define 
\begin{small}
$$\mathcal{W}(J,k) = \rho \big( y_1-\frac{1}{2}k \big) \rho(y_1-\frac{1}{2}(k-1)) \ldots \rho(y_1) 
\ldots \rho(y_2) \ldots \rho(y_2+\frac{1}{2}(k_2-1))\rho(y_2+\frac{1}{2}k_2).$$
\end{small}
We consider the set of pairs $$\Omega=\{(W,k)\mid W\in \{a,b,a^{-1},b^{-1}\}^{<\mathbf{N}}, k \in \mathbf{N}\text{ such that } |W|=2k+1\}.$$

\begin{definition}\label{specialelementsdefinition}
Recall the map $\lambda$ from Definition \ref{SimpleGroup}.
Given an element $f\in F'$ and $\omega=(W,k) \in \Omega$, 
we define the {\em special element} $\lambda_{\omega}(f)\in \textup{Homeo}^+(\mathbf{R})$ as follows: 
For each $n\in \mathbf{Z}$, we let 
$$\lambda_{\omega}(f)\restriction [n,n+1]=
\begin{cases}
\lambda(f)\restriction [n,n+1]\text{ if }  \mathcal{W}([n,n+1],k)=W^{\pm 1}; \\
\id\restriction [n,n+1] \text{ otherwise}.
\end{cases}
$$
\end{definition}

More general special elements were defined in \cite{uniformlyperfect}, but this subclass is the only one we will need for our purposes.

\begin{proposition}[{\cite[Proposition 3.5]{uniformlyperfect}}]
\label{Specialelements}
Let $\rho:\frac{1}{2}\mathbf{Z}\to \{a,a^{-1},b,b^{-1}\}$ be a quasi-periodic labelling.
Let $\omega\in \Omega$ and $f\in F'$. Then $\lambda_{\omega}(f) \in G_{\rho}$.
\end{proposition}

We recall the following \emph{characterization of elements} of $G_{\rho}$.
This provides an alternative, ``global" description of the groups as comprising of elements satisfying dynamical and combinatorial hypotheses,
and is reminiscent of similar descriptions for various generalisations of Thompson's groups. 

\begin{definition}\label{Krho}
Let $K_{\rho}$ be the set
of homeomorphisms $f\in \textup{Homeo}^+(\mathbf{R})$ satisfying the following:
\begin{enumerate}
\item $f$ is a countably singular piecewise linear homeomorphism of $\mathbf{R}$ with a discrete set of singularities, all of which lie in $\mathbf{Z}[\frac{1}{2}]$;
\item The derivatives, whenever they exist, are integer powers of $2$;
\item There is a $k_f\in \mathbf{N}$ such that the following holds:
\begin{enumerate}
\item[3.a] Whenever $x,y\in \mathbf{R}$ satisfy that $$x-y\in \mathbf{Z} \text{ and }  \mathcal{W}(x,k_f)=\mathcal{W}(y,k_f),$$ it holds that $$x-x\cdot f = y-y\cdot f;$$
\item[3.b] Whenever $x,y\in \mathbf{R}$ satisfy that $$x-y\in \mathbf{Z}\text{ and } \mathcal{W}(x,k_f)=\mathcal{W}^{-1}(y,k_f),$$ it holds that $$x-x\cdot f=y'\cdot f-y' \qquad \text{ where }y'=y\cdot \iota.$$
\end{enumerate}
\end{enumerate}
\end{definition}

\begin{theorem}[{\cite[Theorem 1.8]{uniformlyperfect}}]
\label{characterisation}

Let $\rho$ be a quasi-periodic labelling.
The groups $K_{\rho}$ and $G_{\rho}$ coincide.
\end{theorem}

\subsection{Point stabilizers}

Our next goal is to prove a structural result on $G_\rho$ (Proposition \ref{Fprime}), which strengthens a key result from \cite{uniformlyperfect}, and use it together with Theorem \ref{thm:main} to identify subgroups of $G_\rho$ with vanishing second bounded cohomology.

\begin{definition}
\label{def:atom}

A homeomorphism $f\in \textup{Homeo}^+(\mathbf{R})$ is said to be \emph{stable} if there exists an $n\in \mathbf{N}$ 
such that the following holds: For any compact interval $I$ of length at least $n$, there is an integer $m\in I$
such that $f$ fixes a neighbourhood of $m$ pointwise.
Given a stable homeomorphism $f\in \textup{Homeo}^+(\mathbf{R})$ and an interval 
$[m_1,m_2]$, the restriction $f\restriction [m_1,m_2]$ is said to be an \emph{atom of $f$}, if $m_1, m_2 \in \mathbf{Z}$ and $f$ fixes an open neighbourhood of $\{m_1, m_2\}$ pointwise. We will also refer to $[m_1, m_2]$ as the atom: the use of the word will be clear from the context.

Given a stable homeomorphism $f$, we may express $\mathbf{R}$ as a union of intervals with integer endpoints $\{I_{\alpha}\}_{\alpha\in P}$ such that $f\restriction I_{\alpha}$ is an atom for each $\alpha\in P$ and different intervals intersect in at most one endpoint. We call this a \emph{cellular decomposition of $\mathbf{R}$ suitable for $f$}.
\end{definition} 

In \cite{grho} and \cite{uniformlyperfect} the term \emph{atom} refers to atoms in the sense of Definition \ref{def:atom}, which are moreover minimal with respect to inclusion. This definition has the advantage that a cellular decomposition of the real line suitable for a stable homeomorphism $f$ is unique. However, in the proof of Proposition \ref{Fprime} we will need to consider cellular decompositions which are suitable for tuples of elements, so this more relaxed notion will be more convenient.

\begin{definition}
\label{def:conjugateatoms}

Two atoms $f\restriction [m_1,m_2]$ and $f\restriction [m_3,m_4]$ are said to be \emph{conjugate} if $m_2-m_1=m_4-m_3$ and the following holds.
$$f\restriction [m_1,m_2]=h^{-1}\circ f \circ h\restriction [m_3,m_4]\qquad h=T_{[m_1,m_2],[m_3,m_4]};$$
and they are \emph{flip-conjugate} if:
$$f\restriction [m_1,m_2]=h^{-1} \circ f\circ h\restriction [m_3,m_4]\qquad h=T_{[m_1,m_2],[m_3,m_4]}^{or}.$$

Let $f$ be a stable homeomorphism and fix a cellular decomposition $\{I_{\alpha}\}_{\alpha\in P}$ of $\mathbf{R}$ suitable for $f$, and $n \in \mathbf{N}$.
We refine the collection of atoms $\{I_{\alpha}\}_{\alpha\in P}$ into a collection of \emph{decorated atoms}: $$\mathcal{T}_n(f, P)=\{(I_{\alpha},n)\mid \alpha \in P\}.$$
We say that two decorated atoms
$(I_{\alpha}, n)$ and $(I_{\beta}, n)$ are \emph{equivalent} if either of the following holds:
\begin{enumerate}
\item $I_{\alpha},I_{\beta}$ are conjugate and $\mathcal{W}(I_{\alpha},n)=\mathcal{W}(I_{\beta},n)$;
\item $I_{\alpha},I_{\beta}$ are flip-conjugate  and $\mathcal{W}(I_{\alpha},n)=\mathcal{W}^{-1}(I_{\beta},n)$.
\end{enumerate}

The cellular decomposition $\{I_{\alpha}\}_{\alpha\in P}$ of $f$ is said to be \emph{uniform} if there are finitely many equivalence classes of decorated atoms in $\mathcal{T}_n(f, P)$. Note that this definition does not depend on $n$: if there are finitely many equivalence classes in $\mathcal{T}_n(f, P)$, then this is true for any $n\in \mathbf{N}$, since there are finitely many words in $\{a,b,a^{-1},b^{-1}\}^n$.

Let $\zeta$ be an equivalence class of elements in $\mathcal{T}_n(f, P)$.
We define the homeomorphism $f_{\zeta}$ as
$$f_{\zeta}\restriction I_{\alpha}=
\begin{cases}
f\restriction I_{\alpha}\text{ if }(I_{\alpha},n)\in \zeta; \\
f_{\zeta}\restriction I_{\alpha}=\id\restriction I_{\alpha}\text{ if }(I_{\alpha}, n)\notin \zeta.
\end{cases}$$
If $\zeta_1, \ldots ,\zeta_m$ are the equivalence classes of elements in $\mathcal{T}_n (f, P)$,
then the list of homeomorphisms $f_{\zeta_1}, \ldots ,f_{\zeta_m}$ is called the \emph{cellular decomposition of $f$} determined by the cellular decomposition $\{ I_\alpha \}_{\alpha \in P}$ of $\mathbf{R}$ and the integer $n \in \mathbf{N}$. Note that $f = f_{\zeta_1} \cdots f_{\zeta_m}$.
\end{definition}

The main structural result is the following:

\begin{proposition}\label{Fprime}
Let $\rho$ be a quasi-periodic labelling.
Let $f_1, \ldots, f_n\in G_{\rho}$ be elements such that there exists an open interval $I$ which is pointwise fixed by the element $f_i$ for each $1\leq i\leq n$.
Then there is a subgroup $A<G_{\rho}$ isomorphic to a finite direct sum of copies of $F'$, which contains $f_1, \ldots, f_n$.
\end{proposition}

The statement for $n=1$ is contained in the proof of \cite[Proposition 1.9]{uniformlyperfect}. We will prove Proposition \ref{Fprime} without appealing to this.

\begin{proof}
By Theorem \ref{thm:properties}.2, the action of $G_\rho$ is minimal. Therefore there exists an element $h\in G_{\rho}$ such that $0\in I\cdot h$, and hence $f_i^h$ pointwise fixes a neighborhood of $0$ for each $1\leq i\leq n$. It suffices to show the statement for $f_1^h,\ldots,f_n^h$, so we may assume without loss of generality that $I$ is a neighbourhood of $0$.

Let $k := \max\{k_{f_1},\ldots, k_{f_n}\} + 1$ as in Definition \ref{characterisation}. Then for every $x, y \in \mathbf{R}$ such that $x - y \in \mathbf{Z}$ and $\mathcal{W}(x, k - 1) = \mathcal{W}(y, k - 1)$ it holds that
$$x - x \cdot f_i = y - y \cdot f_i, \quad \forall 1\leq i\leq n.$$
If instead $\mathcal{W}(x, k-1) = \mathcal{W}^{-1}(y, k-1)$, then
$$x - x \cdot f_i = y' \cdot f_i - y', \quad \text{where }1\leq i\leq n\text{ and } y' = y \cdot \iota.$$
In particular, if $n \in \mathbf{Z}$ satisfies $\mathcal{W}(n, k) = \mathcal{W}(0, k)$, then each element $f_1,\ldots,f_n$  pointwise fixes a neighbourhood of $n$ (we took $k$ instead of $k - 1$ in order to ensure that this holds also for a left neighbourhood of $n$). We set $\mathcal{N}$ to be the set of such $n \in \mathbf{Z}$, and note that by the definition of quasi-periodic labelling, $\mathcal{N}$ is infinite and the set of distances between two consecutive elements of $\mathcal{N}$ is bounded.

This determines a decomposition of $\mathbf{R}$ into intervals $\{ I_\alpha \}_{\alpha \in P}$ with endpoints in $\mathcal{N}$, and also satisfying that no interval $I_{\alpha}$ contains a point in $\mathcal{N}$. Each element in the list $f_1,\ldots,f_n$ pointwise fixes a neighbourhood of each element of $\mathcal{N}$, and each $I_\alpha$ is an atom for $f_1,\ldots,f_n$. So $\{ I_\alpha \}_{\alpha \in P}$ is a cellular decomposition of $\mathbf{R}$ which is suitable for all the elements $f_1,\ldots,f_n$. Moreover, since two consecutive elements of $\mathcal{N}$ are at a bounded distance from each other, it follows that $\{ |I_\alpha| \}_{\alpha \in P}$ is bounded. In particular the integer $l := k + \max \{ |I_\alpha| : \alpha \in P \}$ is bounded, and the cellular decomposition $\{ I_\alpha \}_{\alpha \in P}$ is uniform: there are finitely many equivalence classes of decorated atoms in $$\mathcal{T}_l(f_1, P) =\mathcal{T}_l(f_2, P)=\cdots= \mathcal{T}_l(f_n, P).$$
We denote by $\zeta_1,\ldots,\zeta_m$ these equivalence classes, and let $\{f_{i,\zeta_j}\mid 1\leq j\leq m\}$ be the corresponding cellular decompositions of $f_1,\ldots,f_n$. 

For each $1\leq j\leq m$, set $L_j := |I_{\alpha}|$ where $(I_{\alpha},l)\in \zeta_j$: this is well-defined since $|I_{\alpha}|=|I_{\beta}|$ whenever $(I_{\alpha},l)$ and $(I_{\beta},l)$ are in the same equivalence class.
Also for each $1\leq j\leq m$, define the canonical isomorphism $$\phi_j:F'\to F_{[0,L_j]}'$$
where $F_{[0,L_j]}$ is the standard copy of $F$ supported on the interval $[0,L_j]$.

For each $1\leq j\leq m$, we have $$\{\mathcal{W}(I_{\alpha},l)\mid (I_{\alpha},l)\in \zeta_j\}=\{W_j,W^{-1}_j\}$$
for some words $W_1, \ldots, W_m$.
Define a map
$$\phi: \bigoplus_{1\leq j\leq m} F'\to \textup{Homeo}^+(\mathbf{R})$$
as follows. 
For $\alpha\in P$ and $1 \leq j \leq m$:
$$\phi(g_1,\ldots,g_m) \restriction I_{\alpha} \cong_T\phi_j(g_j) \qquad \text{ if }(I_{\alpha},l) \in \zeta_j\text{ and } \mathcal{W}(I_{\alpha},l)=W_j;$$
$$\phi(g_1,\ldots,g_m) \restriction I_{\alpha} \cong_{T^{or}}\phi_j(g_j) \qquad \text{ if }(I_{\alpha},l) \in \zeta_j\text{ and } \mathcal{W}(I_{\alpha},l)=W_j^{-1}.$$

This is an injective group homomorphism, and moreover the image of each element satisfies Definition \ref{Krho} with the uniform constant $l$. This implies that the image of $\phi$ is a subgroup $A < G_\rho$, which is isomorphic to a direct sum of $m$ copies of $F'$.
Moreover, for every $1 \leq j \leq m, 1\leq i\leq n$, there is an element $g_{i,j} \in F'$ such that
$$\phi(\id, \ldots, g_{i,j}, \ldots, \id) = f_{i,\zeta_j}.$$
Therefore $f_{i,\zeta_j} \in A$ for each $1\leq i\leq n, 1\leq j\leq m$, and thus $f_i = f_{i,\zeta_1} \cdots f_{i,\zeta_m} \in A$ for each $1\leq i\leq n$.
This concludes the proof.
\end{proof}

We are ready to identify subgroups of $G_\rho$ with vanishing second bounded cohomology.

\begin{proposition}\label{prop:stabbac}
If $\Gamma \leq G_\rho$ admits a global fixpoint $x \in \mathbf{R}$, then $\HH^2_b(\Gamma) = 0$.
\end{proposition}

\begin{proof}
Since $x$ is a global fixpoint of $\Gamma$, and the left and right slope at $x$ are powers of $2$, we obtain a germ homomorphism $\Gamma \to \mathbf{Z}^2$, whose kernel is the subgroup $\Gamma_1$ consisting of elements $f\in \Gamma$ that pointwise fix some open neighbourhood $I_f$ of $x$. By Theorems \ref{thm:amenable} and \ref{thm:amenableext}, it suffices to show that $\HH^2_b(\Gamma_1) = 0$.

Every finitely generated subgroup $\Delta \leq \Gamma_1$ pointwise fixes a nonempty open interval which contains $x$. Therefore by Proposition \ref{Fprime}, $\Delta$ is isomorphic to a subgroup of a finite direct sum of copies of $F'$. In particular, $\Delta$ is isomorphic to a subgroup of $\textup{PL}^+([0, 1])$ and so $\HH^2_b(\Delta) = 0$ by Theorem \ref{thm:PLPP}. Thus $\HH^2_b(\Gamma_1) = 0$ by Proposition \ref{prop:dirun}.
\end{proof}

\subsection{Proof of Theorem \ref{theorem:mainNavas}}

The following proposition is the key idea behind the proof of the Theorem \ref{theorem:mainNavas}.

\begin{proposition}
\label{prop:mainNavas}
Let $\rho$ be a quasi-periodic labelling. Let $a_1, a_2, a_3 \in G_\rho$. Then there exist $h, g_1, g_2, g_3 \in G_\rho$ and non-empty sets $I_1 \subset I_2 \subset I_3 \subset \mathbf{R}$ with the following properties. Let $f_i := a_i^h$ for $i = 1, 2, 3$.
\begin{enumerate}
    \item Each of the groups $\langle g_1, g_2, g_3 \rangle$ and $\langle f_i, g_i \rangle : i = 1, 2, 3$ has a global fixpoint in $\mathbf{R}$.
    \item The element $f_i g_i$ fixes $I_i$ pointwise, for $i = 1, 2, 3$.
    \item $Supp(g_i) \subset I_{i+1}$ for $i = 1, 2$.
\end{enumerate}
\end{proposition}

Let us see how this proposition implies Theorem \ref{theorem:mainNavas}:

\begin{proof}[Proof of Theorem \ref{theorem:mainNavas}]
Let $\alpha \in \HH^2_b(G_\rho)$, and let $\omega$ be the unique homogeneous representative given by Theorem \ref{thm:bouarich}. This is associated to a central extension $E$ and a section $\sigma : G_\rho \to E$ such that $\omega(f, g) = \sigma(f)\sigma(g)\sigma(fg)^{-1}$. By Corollary \ref{cor:bouarich}, if $\Gamma \leq G_\rho$ satisfies $\HH^2_b(\Gamma) = 0$, then the restriction $\sigma|_\Gamma$ is a homomorphism. In particular, this holds when $\Gamma$ is abelian, by Theorem \ref{thm:amenable}, or has a global fixpoint, by Proposition \ref{prop:stabbac}.

We need to show that $\omega \equiv 0$, equivalently that $\sigma$ is a homomorphism. As in the proof of Theorem \ref{thm:main}, this amounts to showing that whenever $a_1, a_2, a_3 \in G_\rho$ satisfy $a_1 a_2 a_3 = \id$, we also have $\sigma(a_1)\sigma(a_2)\sigma(a_3) = \id$. \\

So let $a_1, a_2, a_3$ be as above, and let $h, g_1, g_2, g_3$ and $I_1, I_2, I_3$ be given by Proposition \ref{prop:mainNavas}. We similarly set $f_i = a_i^h$ for $i = 1, 2, 3$, and notice that $f_1 f_2 f_3 = \id$, and that by Lemma \ref{lem:conjugacy} it suffices to show that $\sigma(f_1)\sigma(f_2)\sigma(f_3) = \id$.

\begin{claim}
It holds
$$\sigma(f_1)\sigma(f_2)\sigma(f_3) \cdot \sigma(g_3) \sigma(g_2) \sigma(g_1) = \sigma(f_1 g_1) \sigma(f_2 g_2) \sigma(f_3 g_3).$$
\end{claim}

\begin{proof}[Proof of Claim]
By Proposition \ref{prop:stabbac} we have $\HH^2_b(\langle f_i, g_i \rangle) = 0$. Therefore:
$$\sigma(f_1)\sigma(f_2) \cdot \sigma(f_3) \sigma(g_3) \cdot \sigma(g_2) \sigma(g_1) = \sigma(f_1)\sigma(f_2) \cdot \sigma(f_3 g_3) \cdot \sigma(g_2) \sigma(g_1).$$
Now $f_3 g_3$ fixes $I_3$ pointwise, and $Supp(g_2)\subset I_3$, therefore $\langle f_3 g_3, g_2 \rangle$ is abelian and:
$$\sigma(f_1)\sigma(f_2) \cdot \sigma(f_3 g_3) \sigma(g_2) \cdot \sigma(g_1) = \sigma(f_1)\sigma(f_2) \cdot \sigma(g_2) \sigma(f_3 g_3) \cdot \sigma(g_1).$$
Similarly, $Supp(g_1)\subset I_2 \subset I_3$, therefore $\langle f_3 g_3, g_1 \rangle$ and $\langle f_2 g_2, g_1 \rangle$ are abelian. Thus:
\begin{align*}
    \sigma(f_1)\sigma(f_2) \sigma(g_2) \cdot \sigma(f_3 g_3) \sigma(g_1) &= \sigma(f_1)\sigma(f_2) \sigma(g_2) \cdot \sigma(g_1) \sigma(f_3 g_3) =\sigma(f_1) \cdot \sigma(f_2g_2) \sigma(g_1) \cdot \sigma(f_3 g_3) \\
    &= \sigma(f_1) \cdot \sigma(g_1) \sigma(f_2 g_2) \cdot \sigma(f_3g_3) = \sigma(f_1 g_1)\sigma(f_2 g_2)\sigma(f_3 g_3),
\end{align*}
which concludes the proof of the claim.
\end{proof}

Now the group $\langle f_1 g_1, f_2 g_2, f_3 g_3 \rangle$ fixes $I_1$ pointwise, which is a non-empty set. By Proposition \ref{prop:stabbac}, its second bounded cohomology vanishes. Therefore
$$\sigma(f_1 g_1) \sigma(f_2 g_2) \sigma(f_3 g_3) = \sigma(f_1 g_1 \cdot f_2 g_2 \cdot f_3 g_3).$$
Using the same argument as in the proof of the claim in reverse, we obtain
$$f_1 g_1 \cdot f_2 g_2 \cdot f_3 g_3 = f_1 f_2 f_3 \cdot g_3 g_2 g_1 =  g_3 g_2 g_1.$$
Therefore
$$\sigma(f_1 g_1 \cdot f_2 g_2 \cdot f_3 g_3) = \sigma(g_3 g_2 g_1) = \sigma(g_3) \cdot \sigma(g_2) \cdot \sigma(g_1),$$
where we used that $\HH^2_b(\langle g_1, g_2, g_3 \rangle) = 0$, by Proposition \ref{prop:stabbac}.

We have thus shown that
$$\sigma(f_1)\sigma(f_2)\sigma(f_3) \cdot \sigma(g_3) \sigma(g_2) \sigma(g_1) = \sigma(g_3) \sigma(g_2) \sigma(g_1).$$
This implies that $\sigma(f_1)\sigma(f_2)\sigma(f_3) = \id$, which is what we wanted to show.
\end{proof}

We deduce:

\begin{corollary}
\label{cor:integral}

$\HH^2_b(G_\rho; \mathbf{Z}) = 0$. In particular, every action of $G_\rho$ on the circle has a global fixpoint.
\end{corollary}

\begin{proof}
The inclusion $\mathbf{Z} \subset \mathbf{R}$ induces a long exact sequence \cite[Proposition 1.1]{gersten} (see also \cite[Proposition 8.2.12]{monod}) that begins as follows:
$$0 \to \HH^1(G_\rho; \mathbf{R}/\mathbf{Z}) \to \HH^2_b(G_\rho; \mathbf{Z}) \to \HH^2_b(G_\rho; \mathbf{R}) \to \cdots$$
Since $G_\rho$ is perfect, we have $\HH^1(G_\rho; \mathbf{R}/\mathbf{Z}) \cong \Hom(G_\rho; \mathbf{R}/\mathbf{Z}) = 0$. Therefore Theorem \ref{theorem:mainNavas} implies that $\HH^2_b(G_\rho; \mathbf{Z}) = 0$.
The statement about circle actions is a direct consequence of Ghys's Theorem \cite{ghys} (see also \cite[Proposition 10.20]{BC}).
\end{proof}

The rest of this section will be devoted to the proof of Proposition \ref{prop:mainNavas}.

Let $a_1, a_2, a_3 \in G_\rho$. By Theorem \ref{thm:properties}.1, each $a_i$ fixes a point in $\mathbf{R}$. Therefore there exists a compact interval $J_1 \subset \mathbf{R}$ with endpoints in $\mathbf{Z}[\frac{1}{2}]$ such that each $a_i$ fixes a point inside $J_1$. Next, choose compact intervals $J_2, J_3, J_4$ with endpoints in $\mathbf{Z}[\frac{1}{2}]$ such that
$$J_i \cup J_i \cdot a_i \subset J_{i+1} : 1 \leq i \leq 3.$$

By Theorem \ref{thm:properties}.3 and the definition of a proximal action, there exists an element $h \in G_\rho$ such that $J_4 \cdot h \subset (0, 1)$. We set $T_i := J_i \cdot h$, and note that $T_i \subset (0, 1)$ has endpoints in $\mathbf{Z}[\frac{1}{2}]$, for $i = 1, \ldots, 4$, since $G_\rho$ preserves $\mathbf{Z}[\frac{1}{2}]$.

As in the statement of Proposition \ref{prop:mainNavas}, we set $f_i := a_i^h$. By Definition \ref{Krho} and Theorem \ref{characterisation}, there exists $k \in \mathbf{N}$ such that whenever $x\in [0,1], y \in [n,n+1]$ satisfy $x - y\in \mathbf{Z}$ and $\mathcal{W}([0, 1], k) = \mathcal{W}([n, n+1], k)$, it holds that:
$$x - x \cdot f_i = y - y \cdot f_i : i = 1, 2, 3.$$
Fix $m \neq 0$ such that $\mathcal{W}([0,1], k) = \mathcal{W}([m,m+1], k)$: such an $m$ must exist because $\rho$ is quasi-periodic. However $\rho$ is not periodic by Lemma \ref{lem:notperiodic}, therefore there exists $l > k$ such that $\mathcal{W}([0,1], l) \neq \mathcal{W}([m, m+1], l)$. Moreover, since $\mathcal{W}([0,1], k) = \mathcal{W}([m,m+1], k)$, we also have $\mathcal{W}([0,1], l) \neq \mathcal{W}^{-1}([m,m+1], l)$.

We set $W := \mathcal{W}([0,1], l), \omega := (W, l) \in \Omega$ and $$\mathcal{N}_1 := \{ n \in \mathbf{Z} : \mathcal{W}([n,n+1], l) = W \}\qquad \mathcal{N}_2 := \{ n \in \mathbf{Z} : \mathcal{W}([n,n+1], l) = W^{-1} \}.$$
Note that by construction $m \notin \mathcal{N}_1\cup \mathcal{N}_2$. Finally, we set $$I_i := (T_i + \mathcal{N}_1)\bigcup (T_i\cdot \iota+\mathcal{N}_2).$$ for $i = 1, \ldots, 4$. Then these sets satisfy $\emptyset \subset I_1 \subset I_2 \subset I_3 \subset I_4 \subset \mathbf{R}$, and all of these inclusions are strict. \\

We will now use $\omega$ to define the elements $g_i$: these will be special elements as in Definition \ref{Specialelements}. For $i = 1, 2, 3$, let $\alpha_i$ be an element of $F$ supported on $T_{i+1} \subset (0, 1)$ such that $\alpha_i \restriction T_i\cdot f_i = f_i^{-1} \restriction T_i \cdot f_i$. This is possible since $T_i \cdot f_i \subset T_{i+1} \subset (0, 1)$ for all $i = 1, 2, 3$, and the restriction $f_i^{-1} \restriction T_i \cdot f_i$ is piecewise linear with breakpoints in $\mathbf{Z}[\frac{1}{2}]$ and slopes (when they exist) equal to powers of $2$. Note that $\alpha_i \in F'$, and so we can define $g_i := \lambda_\omega(\alpha_i)$. \\

Now that we have defined all objects involved, we will prove that each item of Proposition \ref{prop:mainNavas} holds.

\begin{claim}
\label{claim:support}

For $i = 1, 2, 3$, we have $Supp(g_i) \subset I_{i+1}$. In particular, $\langle g_1, g_2, g_3 \rangle$ has a global fixpoint.
\end{claim}

\begin{proof}
By construction, $\alpha_i \in F$ is supported on $T_{i+1} \subset (0, 1)$. Let $x \in \mathbf{R}$ be such that $x \cdot g_i \neq x$, and let $n \in \mathbf{Z}$ be such that $x \in [n, n+1]$. Since $g_i = \lambda_\omega(\alpha_i)$, it holds that either: $\mathcal{W}([n, n+1], l) = W$ or $\mathcal{W}([n, n+1], l) = W^{- 1}$. Then by definition, the following holds. In the former case, $n \in \mathcal{N}_1$ and $x \in T_{i+1} + n \subset I_{i+1}$. In the latter case, $n \in \mathcal{N}_2$ and $x \in T_{i+1}\cdot \iota + n \subset I_{i+1}$. This shows that $Supp(g_i) \subset I_{i+1}$.

Now $\langle g_1, g_2, g_3 \rangle$ is supported on $I_4$, and so every point in $\mathbf{R} \setminus I_4 \neq \emptyset$ is a global fixpoint. (Note that $\mathbf{Z}$ will, in particular, be fixed pointwise.)
\end{proof}

\begin{claim}
\label{claim:fixpoint}

For $i = 1, 2, 3$, the group $\langle f_i, g_i \rangle$ has a global fixpoint, and $f_i g_i$ fixes $I_i$ pointwise.
\end{claim}

\begin{proof}
Fix $i \in \{1, 2, 3\}$. Recall that we fixed an element $m \in \mathbf{Z}$ such that $m \notin \mathcal{N}_1\cup \mathcal{N}_2$ (that is $\mathcal{W}([m,m+1], l) \neq W^{\pm 1}$), but $\mathcal{W}([m, m+1], k) = \mathcal{W}([0, 1], k)$. Moreover, recall that $f_i$ has a fixpoint $x_i$ inside $T_1$, because $a_i$ has a fixpoint inside $J_1$. It follows from the definition of $k$ that $m + x_i \in [m, m+1]$ is a fixpoint of $f$. Moreover, from Claim \ref{claim:support}, it follows that $g_i$ fixes $[m, m+1]$ pointwise. Therefore $m + x_i$ is a global fixpoint of $\langle f_i, g_i \rangle$.

By construction, $$f_i^{-1} \restriction T_i \cdot f_i = \alpha_i \restriction T_i \cdot f_i = g_i \restriction T_i \cdot f_i$$ Therefore, $f_i g_i$ fixes $T_i$ pointwise.

Now let $x\in I_i$. We wish to show that $x\cdot f_ig_i=x$. Let $n\in \mathbf{Z}$ be such that $x\in [n,n+1)$.
Then either $n\in \mathcal{N}_1$ and $x\in T_i+n$, or $n\in \mathcal{N}_2$ and $x\in T_i\cdot \iota+n$.

By definition of $l$, the fact that $l>k$, and construction of the $g_i$, we have the following. If $x$ is a fixpoint of $f_i g_i$ and $y \in \mathbf{R}$ satisfies $x - y = n \in \mathcal{N}_1$, then $y$ is also a fixpoint of $f_i g_i$. Similarly, if $x$ is a fixpoint of $f_i g_i$ and $y \in \mathbf{R}$ satisfies $x - y\cdot \iota = n \in \mathcal{N}_2$, then $y$ is also a fixpoint of $f_i g_i$. Our claim follows.
\end{proof}

\begin{proof}[Proof of Proposition \ref{prop:mainNavas}]
Item $1$ follows from Claims \ref{claim:support} and \ref{claim:fixpoint}, Item $2$ from Claim \ref{claim:fixpoint} and Item $3$ from Claim \ref{claim:support}.
\end{proof}

\subsection{Piecewise linear homeomorphisms of flows}
\label{ss:nico}

We end by explaining an alternative approach to our results, which is more conceptual and applies to the groups $T(\varphi, \sigma)$ from \cite{LBMB}. We will only sketch the definitions and outline the argument, and refer the reader to \cite{MBT, LBMB} throughout for further details. \\

Let $X$ be a nonempty Stone space (that is, a totally disconnected compact Hausdorff space), and let $\varphi$ be a homeomorphism of $X$, such that the induced action of $\mathbf{Z}$ on $X$ is free. Let $Y^\varphi := (X \times \mathbf{R}) / \mathbf{Z}$, where $\mathbf{Z}$ acts on $X \times \mathbf{R}$ diagonally via $\varphi \times (x \mapsto x+1)$. Given a clopen set $C \subset X$, and an open dyadic interval $J \subset \mathbf{R}$ small enough so that the inclusion $C \times J \subset X \times \mathbf{R}$ stays injective after taking the quotient, we call the induced embedding $\pi_{C \times J} : C \times J \to Y^\varphi$ a \emph{dyadic chart}. The group $T(\varphi)$ is defined as the group of all those homeomorphisms $g$ of $Y^\varphi$ that are isotopic to the identity, and that are piecewise given by orientation-preserving dyadic homeomorphisms on dyadic charts. More precisely for every $y \in Y^\varphi$ there exists a dyadic chart $C \times J$ and an orientation-preserving dyadic homeomorphism $f : J \to I$ such that $g(\pi_{C \times J}(C \times J)) = \pi_{C \times I}(C \times I)$, and $g$ takes the form $(x, t) \mapsto (x, f(t))$ on those charts.

Now let $\sigma$ be a homeomorphism of $X$ such that $\sigma^2 = \id$ and $\sigma \varphi \sigma = \varphi^{-1}$. This defines an action of $D_\infty$ on $X$, which we assume to be free, and similarly defines a space $Y^{\varphi, \sigma} := (X \times \mathbf{R}) / D_\infty$. We analogously define \emph{dyadic charts} $\pi_{C, J} : C \times J \to Y^{\varphi, \sigma}$ for pairs $C, J$ such that this map is an embedding, and denote by $T(\varphi, \sigma)$ the group of homeomorphisms of $Y^{\varphi, \sigma}$ that are isotopic to the identity, and that are piecewise given by dyadic homeomorphisms on dyadic charts. \\

The groups $T(\varphi)$ are left orderable \cite[Proposition 3.3]{MBT}, and $T(\varphi, \sigma)$ embeds into $T(\varphi)$ \cite[Proposition 4.12]{LBMB}, so $T(\varphi, \sigma)$ is also left orderable. Moreover, if $\varphi$ is conjugate to a subshift, then $T(\varphi)$ is finitely generated \cite[Theorem A]{MBT}; and if $\varphi$ is minimal, then $T(\varphi)$ is simple \cite[Theorem B]{MBT}. In \cite[Remark 4.18]{LBMB}, the authors point out that these results can be generalized to the groups $T(\varphi, \sigma)$ without much effort. Moreover, in \cite[Remark 4.19]{LBMB}, they explain that, if $\rho$ is a quasi-periodic labelling, then one can define a subshift $\varphi$ in terms of $\rho$, and $\sigma$ to be a sort of formal inversion on $\rho$, to obtain $G_\rho \cong T(\varphi, \sigma)$. \\

Let us now explain how to prove that $\HH^2_b(T(\varphi, \sigma)) = 0$ by using the action on $Y^{\varphi, \sigma}$ (in place of the action of $G_\rho$ on the line). For every $y \in Y^{\varphi, \sigma}$, as is remarked in the proof of \cite[Theorem 4.13]{LBMB}, every finitely generated group that fixes pointwise a neighbourhood of $y$ is isomorphic to a subgroup of a finite direct sum of copies of $F'$: this is the analogue of our Proposition \ref{Fprime}. By using that the germs at $y$ are amenable (which follows from the piecewise linear nature of these groups), this implies as in Proposition \ref{prop:stabbac} that every subgroup of $T(\varphi, \sigma)$ that admits a global fixpoint in $Y^{\varphi, \sigma}$ has vanishing second bounded cohomology.

To conclude, we need to prove an analogue of Proposition \ref{prop:mainNavas}: the rest of the proof goes through analogously. The natural way to adapt the statement is by replacing $\mathbf{R}$ with $Y^{\varphi, \sigma}$ and intervals with dyadic charts. This can be achieved similarly to our construction. The two main properties of $T(\varphi, \sigma)$, we use are the following. First, every element in $T(\varphi, \sigma)$ fixes a point in $Y^{\varphi, \sigma}$: this follows from \cite[Lemma 4.16(ii)]{LBMB}. Secondly, the action of $T(\varphi, \sigma)$ on $Y^{\varphi, \sigma}$ is extremely proximal: this is contained in \cite[Proposition 4.17]{LBMB}).

Therefore we may proceed as in our proof of Proposition \ref{prop:mainNavas}. Let $a_1, a_2, a_3 \in T(\varphi, \sigma)$. Each fixes a point in $Y^{\varphi, \sigma}$, therefore by proximality there exists a dyadic chart $J_1 \subset Y^\varphi$ such that each $a_i$ fixes a point in $J_1$. Then we can choose the sets $J_2, J_3, J_4$ analogously, namely such that $J_i \cup J_i \cdot a_i \subset J_{i+1}$. Using proximality again, we may replace the $a_i$ by conjugates $f_i := a_i^h$ and assume that each $J_i$ is a small dyadic chart. This allows to define piecewise elements $g_1, g_2, g_3$ so that $g_i$ is supported on $J_{i+1}$ and coincides with $f_i^{-1}$ on $J_i$. Then our same proof shows that the $g_i$ satisfy the conclusions of Proposition \ref{prop:mainNavas}.

This argument then also encompasses the case of the groups $G_\rho$ by \cite[Remark 4.19]{LBMB}. \\

As for the groups $T(\varphi)$, the argument does not carry over, since it is not true that every element in $T(\varphi)$ fixes a point in $Y^{\varphi}$ (although the other properties that we used about $T(\varphi, \sigma)$ do hold for $T(\varphi)$). Thus, we do not know whether $\HH^2_b(T(\varphi)) = 0$ or not.

\section{Groups acting on the circle}
\label{s:circle}

Having gathered a good understanding of the second bounded cohomology of groups acting on the interval and on the line, we move to the study of groups acting on the circle. Most interesting actions lead to non-trivial real Euler classes (Theorem \ref{thm:burger}), so one cannot expect a general vanishing result, but Theorem \ref{thm:T} is the next-best thing. In order to prove it we use the following general result (see also \cite{MonodNariman}):

\begin{proposition}[{\cite[Proposition 6.9]{bcbinate}}]
\label{prop:circle}
Let $G$ be a group of orientation-preserving homeomorphisms of the circle. Suppose that $G$ has an orbit $Y$ such that:
\begin{enumerate}
    \item For $k = 2, 3$ the action of $G$ on circularly ordered $k$-tuples in $Y$ is transitive;
    \item For $k = 1, 2$ the stabilizer $H_k$ of a $k$-tuple in $Y$ satisfies $\HH^2_b(H_k) = 0$.
\end{enumerate}
Then $\HH^2_b(G)$ is one-dimensional, spanned by the Euler class.
\end{proposition}

Using the results from the previous section, Theorem \ref{thm:T} is then an immediate consequence:

\begin{proof}[Proof of Theorem \ref{thm:T}]
Let $G$ be a group of piecewise linear (or piecewise projective) homeomorphisms of the circle. Then the stabilizer of a $k$-tuple, for all $k \geq 1$, can be realized as a group of piecewise linear (or piecewise projective) homeomorphisms of the interval (or the real line).
Now suppose that $G$ satisfies the hypotheses of Theorem \ref{thm:T}, that is, $G$ has an orbit $Y$ such that the action on circularly ordered pairs and triples in $Y$ is transitive. Then the results follows by combining Theorem \ref{thm:PLPP} with Proposition \ref{prop:circle}.
\end{proof}

Theorem \ref{thm:T} was already known for Thompson's group $T$ \cite{spectrum1}, but the only available proofs made use of the computation of its cohomology \cite{ghys_serg}. The general statement of Theorem \ref{thm:T} makes it applicable to several other groups, most notably to certain Stein--Thompson groups, to the finitely presented infinite simple group $S$ constructed by the second author in \cite{lodhaS} and to certain irrational-slope analogues of Thompson's group, which will be the subject of the next section.
The special case of Stein--Thompson groups had been conjectured by Heuer and L\"oh \cite[Conjecture A.5]{spectrumv1}. See \cite{monsters} for a detailed discussion of these examples. \\

Recall the notion of \emph{proximal action} from Subsection \ref{ss:dynamics}. Combining Theorem \ref{thm:T} with Corollary \ref{cor:onedim}, we obtain:

\begin{corollary}\label{cor:rigidity}
Let $G$ be a group satisfying the hypotheses of Theorem \ref{thm:T}. Suppose that the orbit $Y$ is moreover dense. Then $G$ admits a unique proximal action on the circle, up to conjugacy.
\end{corollary}

\begin{proof}
This is a combination of the aforementioned results: an action which is transitive on circularly ordered pairs of a dense subset $Y$ of the circle is easily seen to be proximal. Indeed, let $I, J$ be open subset of the circle such that the closure of $J$ is proper and compact. Let $K$ a closed arc with endpoints in $Y$ that contains $J$, and let $L$ be a closed arc with endpoints in $Y$ that is contained in $I$. Using the hypothesis on double transitivity, we obtain an element $f \in G$ such that $J \cdot f \subset K \cdot f = L \subset I$.
\end{proof}

With an application of Gromov's Mapping Theorem (Theorem \ref{thm:mapping}), the computation of the second bounded cohomology of lifts of groups acting on the circle is immediate:

\begin{corollary}\label{cor:lift}
Let $G$ be as in Proposition \ref{prop:circle}, and let $\overline{G}$ be a lift of $G$ to the real line. Then $\HH^2_b(\overline{G})$ is one-dimensional, spanned by the rotation quasimorphism.
\end{corollary}

\begin{proof}
The group $\overline{G}$ is a central extension of $G$, so by Theorem \ref{thm:mapping} the quotient $\overline{G} \to G$ induces an isomorphism in bounded cohomology. By \cite[Proposition 10.26]{BC}, the real Euler class is sent to the image of the rotation quasimorphism in second bounded cohomology.
\end{proof}

In particular, we are able to reduce the computation of the stable commutator length in such groups to an evaluation of the rotation quasimorphism:

\begin{corollary}\label{cor:sclrot}
Let $G$ be a group satisfying the hypotheses of Theorem \ref{thm:T}. Suppose that the orbit $Y$ is moreover dense.
Then
$\scl(g) = \frac{\rot(g)}{2}$
for all $g \in \overline{G}'$.
\end{corollary}

\begin{proof}
By Corollary \ref{cor:lift} and Bavard duality, if $D \geq 0$ is the defect of the rotation quasimorphism of $\overline{G}$, we have $\scl(g) = \rot(g)/2D$, for all $g \in \overline{G}'$. Therefore it suffices to show that $D = 1$.
Recall from Example \ref{ex:rot} that $\rot$ has defect $1$ on $\overline{\Homeo^+(\mathbf{R}/\mathbf{Z})}$ \cite[Proposition 2.92]{calegari}. Therefore $D \leq 1$, and it remains to prove that $D \geq 1$.

Let $(a, b, c, d)$ be a circularly ordered $4$-tuple in $Y$. By double transitivity there exist elements $g, h \in G$ such that $(b, c) \cdot g = (a, d)$ and $(c, d) \cdot h = (b, a)$. It follows that
$$a \cdot [g, h] = a \cdot g^{-1} h^{-1} g h = b \cdot h^{-1} g h = c \cdot gh = d \cdot h = a.$$
Choosing lifts $\overline{g}, \overline{h} \in \overline{G}$ with $\rot(\overline{g}), \rot(\overline{h}) \leq 1$, we obtain an element $\overline{a} \in \mathbf{R}$ such that $\overline{a} \cdot [\overline{g}, \overline{h}]  = \overline{a} + 1$. It follows that $\rot([\overline{g}, \overline{h}]) = 1$. Since a homogeneous quasimorphism evaluates to at most the defect on a commutator \cite[2.2.3]{calegari}, we conclude.
\end{proof}

%This corollary will be used in the next section to construct elements with algebraic irrational stable commutator length in finitely presented groups.

\section{The group $T_{\tau}$ and algebraic irrational \\ stable commutator length}
\label{s:scl}

Let $\tau := \frac{\sqrt{5} - 1}{2}$ be the small golden ratio. In \cite{cleary}, the author defines the \emph{golden ratio Thompson group} $F_\tau$ as the group of orientation-preserving piecewise linear homeomorphisms $g$ of $[0, 1]$ such that the following two properties are satisfied:
\begin{enumerate}
    \item $g$ has finitely many breakpoints, all of which lie in $\mathbf{Z}[\tau]$;
    \item The derivatives, whenever they exist, are integer powers of $\tau$.
\end{enumerate}

Note that it automatically follows that $F_\tau$ preserves $\mathbf{Z}[\tau] \cap [0, 1]$. This is the analogous definition as Thompson's group $F$, where $\tau$ plays the role of $\frac{1}{2}$. 
This group was studied combinatorially via tree diagrams in \cite{burilloF}. The group action on the interval corresponding to the tree diagrams defined in \cite{burilloF} coincides with the above action.

We consider the following ``circle" analogue of the group $F_\tau$.

\begin{definition}\label{Definition: main}
Define the \emph{golden ratio Thompson group} $T_\tau$ as the group of orientation-preserving piecewise linear homeomorphisms $g$ of $\mathbf{S}^1 = \mathbf{R}/\mathbf{Z}$ such that the following three properties are satisfied:

\begin{enumerate}
    \item $g$ has finitely many breakpoints, all of which lie in in $\mathbf{Z}[\tau]/\mathbf{Z}$;
    \item The derivatives, whenever they exist, are integer powers of $\tau$;
    \item $g$ preserves $\mathbf{Z}[\tau]/\mathbf{Z}$.
\end{enumerate}
\end{definition}

Note that, unlike with $F_\tau$, it is necessary here to impose the third condition, since every rotation satisfies the first two. Again, this is the analogous definition as Thompson's group $T$, where $\tau$ plays the role of $\frac{1}{2}$. Crucially, $T_\tau$ contains all rotations with angles in $\mathbf{Z}[\tau]$, and the stabilizer of a point $x \in \mathbf{Z}[\tau]/\mathbf{Z}$ is naturally isomorphic to $F_\tau$. \\

The group $T_{\tau}$ has been studied in \cite{burilloT} by combinatorial means.
We now clarify that the definition of $T_{\tau}$ provided in \cite{burilloT}, by means of combinatorial tree diagrams, defines the same group as Definition \ref{Definition: main}. Note that the group action on $\mathbf{R/Z}$ defined as $T_{\tau}$ in page $2$ of \cite{burilloT} is incorrect, since the third condition of our Definition \ref{Definition: main} is omitted. However, the combinatorial model described by tree diagrams in \cite{burilloT} is correct and it corresponds to our definition above.

\begin{lemma}\label{lemma: same group}
The group action of $T_{\tau}$ on $\mathbf{R/Z}$ in Definition \ref{Definition: main} and the group action on $\mathbf{R/Z}$ corresponding to the group $T_{\tau}$ defined by combinatorial tree diagrams in \cite[Section 2]{burilloT} represent the same group action.
\end{lemma}

\begin{proof}
For the sake of this proof, we denote the group action defined in \cite[Section 2]{burilloT} by means of combinatorial tree diagrams as $G$, and we will show that it is the same as $T_{\tau}$ above. Note that the combinatorial tree diagrams defined in \cite{burilloT} readily translate into concrete piecewise linear homeomorphisms of $\mathbf{R/Z}$ and we will use the latter without giving all details. 

The elements of the generating set $\{x_n,y_n,c_n\}$ of $G$ (as provided in \cite[Section 2]{burilloT}) are easily seen to satisfy Definition \ref{Definition: main}, so it follows that $G\leq T_{\tau}$.
Now note that the orbit of $0$ in both group actions $G$ and $T_{\tau}$ is the same since it equals $\mathbf{Z}[\tau]/\mathbf{Z}$. So given an element $f\in T_{\tau}$, there is an element $g\in G$ such that $0\cdot fg^{-1}=0$. In particular, $fg^{-1}\in F_{\tau}$. 

The group action on $\mathbf{R/Z}$ corresponding to the tree diagrams defined in \cite{burilloF} representing elements of $F_{\tau}$ coincides with our definition of $F_{\tau}$, which is also the stabilizer of $0$ in $G$.
It follows that $F_{\tau}\leq G$ and thus $f\in G$. Therefore, $T_{\tau}\leq G$ and $G=T_{\tau}$.
\end{proof}

Our result will make use of the following structural results from \cite{burilloF, burilloT}. Let $F_\tau^c \leq F_\tau$ be the subgroup of elements $g$ for which $\overline{Supp(g)}\subset (0, 1)$.

\begin{proposition}[Burillo--Nucinkis--Reeves {\cite[Proposition 5.2]{burilloF}}]
\label{prop:index2F}
$F_\tau' = (F_\tau^c)'$ is an index-two subgroup of $F_\tau^c$.
\end{proposition}

\begin{proposition}[Burillo--Nucinkis--Reeves {\cite[Theorem 3.2]{burilloT}}]
\label{prop:index2T}

$T_\tau'$ is an index-two subgroup of $T_\tau$.
\end{proposition}

The main result of this section is that the full lift $\overline{T_\tau}$ of $T_\tau$ to the real line has elements with stable commutator length equal to any $x \in \frac{1}{2} \mathbf{Z}[\tau]$, proving Theorem \ref{thm:scl}.
%We will prove this by applying Corollary \ref{cor:sclrot}, but we will also give a direct dynamical proof that $T_\tau'$ is uniformly simple, and explain how this leads to an elementary proof of the same fact that does not rely on the full power of Proposition \ref{prop:circle}.

\subsection{Transitivity properties}

In this subsection we collect some useful facts about the dynamics of the action of $T_\tau$ on $\mathbf{R}/\mathbf{Z}$.
We start by recording a few simple facts about the transitivity properties of $F_\tau$.
For every $a, b \in \mathbf{Z}[\tau] \cap [0, 1]$, let $F_\tau[a, b]$ be the set of elements whose support is contained in $[a, b]$.
It is shown in \cite[Proposition 6.2]{burilloF} that $F_\tau[a, b]$ is naturally isomorphic to $F_\tau$. (In fact, the two actions are topologically conjugate: the conjugating map is a carefully chosen piecewise linear map $[a, b] \to [0, 1]$).

\begin{lemma}
\label{lem:transF}
The action of $F_\tau$ on the set of ordered $n$-tuples $0 < x_1 < \cdots < x_n < 1$ in $\mathbf{Z}[\tau]\cap [0,1]$ is transitive, and the stabilizer of each such $n$-tuple is isomorphic to $F_\tau^{n+1}$. 
\end{lemma}

\begin{proof}
  Since for each $a,b\in \mathbf{Z}[\tau] \cap [0, 1]$, $F_\tau[a, b]$ is naturally isomorphic to $F_\tau$, this implies the statement about the stabilizers. Indeed, by the self similarity feature of the definition, the stabilizer of $\{x_1, \ldots, x_n\}$ splits as a direct product $F_\tau[0, x_1] \times \cdots \times F_\tau[x_n, 1]$.

Double transitivity is proven in \cite[Corollary 1]{cleary}. Finally, high transitivity follows from an elementary inductive argument using transitivity and the previous description of the stabilizers.
\end{proof}

\begin{lemma}\label{lem:nTransPrime}
For each $n\in \mathbf{N}$, the action of $F_{\tau}'$ on ordered 
$n$-tuples in $\mathbf{Z}[\tau] \cap (0,1)$ is transitive. 
\end{lemma}
\begin{proof}
Let $0 < x_1 < \cdots < x_n < 1$ and $0 < y_1 < \cdots < y_n < 1$.
We choose $a \in (0, \min\{x_1, y_1\}) \cap \mathbf{Z}[\tau]$ and $b \in (\max\{x_n, y_n\}, 1) \cap \mathbf{Z}[\tau]$.

By Lemma \ref{lem:transF}, there exists an element $g \in F_\tau$ fixing $a$ and $b$ and sending $x_i$ to $y_i$ for $i = 1, \ldots, n$. Let $h \in F_\tau$ be supported on $[0, a] \cup [b, 1]$ so that
$$g \mid [0, a] \cup [b, 1] = h \mid [0, a] \cup [b, 1].$$
Then $f := gh^{-1} \in F_\tau^c$ fixes $[0, a] \cup [b, 1]$ pointwise and sends $x_i$ to $y_i$, for $i = 1, \ldots, n$.

Using proximality, we find an element $l \in F_\tau$ such that $[a, b]\cdot l\subset (b, 1)$. Then $l^{-1}f^{-1}l$ is supported on $[b, 1]$, and in particular it fixes every $x_i$. Therefore $[l, f] \in F_\tau'$ is the desired element sending $x_i$ to $y_i$ for all $i = 1, \ldots, n$.
\end{proof}

Let us move to $T_\tau$ and prove the analogous properties.

\begin{lemma}
\label{lem:transT}
The action of $T_\tau$ on the set of circularly ordered $n$-tuples in $\mathbf{Z}[\tau]/\mathbf{Z}$ is transitive, and the stabilizer is isomorphic to $F_\tau^n$.
\end{lemma}

\begin{proof}
Since $T_\tau$ contains all rotations by an angle in $\mathbf{Z}[\tau]/\mathbf{Z}$, the action on $\mathbf{Z}[\tau]/\mathbf{Z}$ is transitive. Moreover, the stabilizer of each point $x\in \mathbf{Z}[\tau]/\mathbf{Z}$ is naturally isomorphic to $F_\tau$, so high transitivity now follows from Lemma \ref{lem:transF}, as does the statement on the stabilizers.
\end{proof}

These transitivity results are enough to prove that $T_\tau$ is of type $F_\infty$. 
Recall that a group $G$ is said to be \emph{of type $F_n$} for an integer $n \geq 1$ if it admits a classifying space with compact $n$-skeleton, and \emph{of type $F_\infty$} if it is of type $F_n$ for all $n \geq 1$. Being of type $F_1$ is equivalent to being finitely generated, and being of type $F_2$ is equivalent to being finitely presented. We refer the reader to \cite{brown_book, brown_fp} and \cite{alonso} for more details.
Now we show the following:

\begin{proposition}
\label{prop:Tfin}
The group $T_\tau$ is of type $F_\infty$.
\end{proposition}

\begin{proof}[Proof]
In \cite{cleary}, Cleary showed that $F_\tau$ is of type $F_\infty$. We will use this as a base case to deduce the result.

We construct a simplicial complex $X$, endowed with an action of $T_{\tau}$, as follows: the $n$-simplices are given by $(n+1)$-tuples $(x_0, \ldots, x_n)$ of elements in $\mathbf{Z}[\tau]/\mathbf{Z}$, and the face relation is given by inclusion. This is clearly contractible, and endowed with a natural action of $T_\tau$. We claim that the action is cocompact on simplices of every given dimension, and that the stabilizer of a simplex is of type $F_\infty$.

Recall from Lemma \ref{lem:transT} that $T_\tau$ acts transitively on the set of circularly ordered $n$-tuples in $\mathbf{Z}[\tau]/\mathbf{Z}$. It follows that the action on the $n$-simplices of $X$ has finitely many orbits, and thus it is cocompact in each dimension.
By Lemma \ref{lem:transT} again, the stabilizer of an $n$-simplex is isomorphic to $F_\tau^{n+1}$. In particular it is of type $F_\infty$ since this is a property closed under finite direct sums.

It follows from a special case of Brown's finiteness criterion \cite[Proposition 1.1]{brown_fp} that the group $T_{\tau}$ is of type $F_{\infty}$.
\end{proof}

\subsection{The lift of $T_\tau$}

We will now use the previous results to study the total lift $G = \overline{T_\tau}$ of $T_\tau$, and prove Theorem \ref{thm:scl}. We start by showing that it also has an abelianization of order $2$.

\begin{lemma}\label{lemma: lift}
Let $G$ be the total lift of $T_\tau$. Then $G$ is the unique lift of $T_\tau$ to the real line. The same holds for the total lift $H$ of $T_\tau'$. 
Moreover, it holds that $G'=H$ and $[G:G']=2$.
\end{lemma}

\begin{proof}
We will show the first statement for $H$, a similar argument applies to $G$.
Consider an arbitrary lift $K\leq \overline{\Homeo^+(\mathbf{R/Z})}$ of $T_{\tau}'$.
To show that it is equal to $H$, it suffices to show that $K$ must contain all integer translations.

Recall the map $\eta:\overline{\Homeo^+(\mathbf{R/Z})}\to \Homeo^+(\mathbf{R/Z})$.
Let $x\in \mathbf{Z}[\tau]/\mathbf{Z}$, and let $T_\tau(x)$ denote the subgroup of elements of $T_\tau$ that fix pointwise a neighbourhood of $x$, so $T_\tau(x) \cong F_\tau^c$ and $T_\tau(x)' \cong F_\tau'$ (see Proposition \ref{prop:index2F}).
Let $f_1,f_2\in T_{\tau}(x)'$ and let $\lambda_1,\lambda_2\in K$ be elements such that $\eta(\lambda_i)=f_i$. Since each $f_i$ fixes $x$, each $\lambda_i$ preserves $x+\mathbf{Z}$ in $\mathbf{R}$.
It follows that the element $\lambda=[\lambda_1,\lambda_2]$ fixes $x+\mathbf{Z}$ pointwise, and 
$\eta(\lambda)=[f_1,f_2]$. Since $T_{\tau}(x)'$ is perfect (see e.g. \cite{prox}), so generated by commutators of elements in $T_\tau(x)'$, we conclude that the group $K$ contains the lift $K_x$ of $T_{\tau}(x)'$ that fixes $x+\mathbf{Z}$ pointwise. This holds for each $x\in \mathbf{Z}[\tau]/\mathbf{Z}$. 
It is an elementary exercise to show that the group generated by the groups $\{K_x\mid x\in \mathbf{Z}[\tau]/\mathbf{Z}\}$ contains all integer translations.

Since $G'$ and $H$ are both lifts of $T_{\tau}'$, by the previous paragraph they coincide. Now $H$ is the preimage under $\eta$ of an index-two subgroup of $T_\tau$, so it has index $2$ in $G = \eta^{-1}(T_\tau)$.
\end{proof}

We are finally ready to prove Theorem \ref{thm:scl}. More precisely:

\begin{theorem}\label{thm:scl2}
For every $x \in \frac{1}{2}\mathbf{Z}[\tau]$, there exists $g \in G$ such that $\scl(g) = x$. Moreover, if $x \in \mathbf{Z}[\tau]$, we may choose $g \in G'$.
\end{theorem}

\begin{proof}[Proof]
The group $T_\tau$ is of type $F_\infty$ by Proposition \ref{prop:Tfin}. Being a central extension of $T_\tau$ by $\mathbf{Z}$, also $G$ is of type $F_\infty$. Moreover, by Lemma \ref{lem:transT} and Corollary \ref{cor:sclrot}, for every $g \in G'$ it holds $\scl(g) = \rot(g)/2$.

Now let $0 < \alpha \in \mathbf{Z}[\tau]$. The translation $f: t \mapsto t + \alpha$ on $\mathbf{R}$ descends to a rotation by $\alpha \mod \mathbf{Z}$, which belongs to $T_\tau$. Moreover, by Lemma \ref{lemma: lift} it holds $f^2 \in G'$. Then
$$\scl(f) = \frac{\scl(f^2)}{2} = \frac{\rot(f^2)}{4} = \frac{\alpha}{2}.$$
\end{proof}

Our proof relied on Corollary \ref{cor:sclrot}, thus on Proposition \ref{prop:circle}. However, the full knowledge of the second bounded cohomology of $T_\tau$ was not necessary to carry over the proof. Indeed, it is possible to deduce the same result by only knowing that $\scl$ vanishes identically on $T_\tau$, which can be proven using the structural results about $T_\tau$ we presented here.
More details on this argument can be found in the preprint \cite{algebraicscl}. \\

We have given a concrete answer to Question \ref{qCalegari}, which gives new insight into the spectrum of $\scl$ on finitely presented groups. It is possible that our proof can be generalized to reach other \emph{metallic ratios}, i.e. quadratic irrationals of the form
$$\lambda := \frac{-n \pm \sqrt{n^2 + 4}}{2}.$$
Indeed, as noted by Cleary at the end of \cite{sqrt2}, most of the arguments establishing finite presentability appear to carry over to that setting.
On the other hand, it is not clear how to generalize the arguments to other algebraic integers. 
In \cite{cleary}, Cleary cites unpublished work generalizing the theory of regular subdivisions (carried over therein for the golden ratio) to algebraic rings. However, this work remains unpublished. Therefore we end by asking the following more specialized version of Question \ref{qCalegari}:

\begin{question}
Does every algebraic integer occur as the $\scl$ of a finitely presented group?
\end{question}

\section{Algebraic irrational simplicial volume}
\label{s:simvol}

Our final application concerns the spectrum of simplicial volume, the main result being Theorem \ref{thm:sv}. We start with the definition. For the rest of this section, all manifolds will be assumed to be oriented, closed and connected.

\begin{definition}
Let $M$ be an $n$-manifold, and let $[M] \in \HH_n(M; \mathbf{R})$ be its real fundamental class. The \emph{simplicial volume} of $M$ is defined as:
$$\|M\| := \inf \left\{ \sum |a_i| : \sum a_i \sigma_i \text{ is a cycle representing } [M] \right\}.$$
\end{definition}

This invariant was introduced by Gromov in the seminal paper \cite{gromov}, and despite being a homotopical invariant, it carries a surprising amount of geometric information \cite[Chapter 7]{BC}. However, there are very few cases in which the precise value of $M$ is known. The largest class of explicit examples is provided by the following theorem:

\begin{theorem}[Heuer--L\"oh \cite{spectrum1}]
\label{thm:spectrum}

Let $G$ be a finitely presented group such that $\HH_2(G; \mathbf{R}) = 0$, and let $g \in G'$. Then there exists a $4$-manifold $M$ such that $\|M\| = 48 \cdot \scl(g)$.
\end{theorem}

This allows to prove that all rationals \cite{spectrum1}, as well as certain trascendental numbers \cite{spectrum2}, occur as the simplicial volume of a $4$-manifold. 

The spectrum of $\scl$ for finitely presented groups is better understood than the spectrum of simplicial volume. It was already known that all rationals \cite{ghys_serg} and certain trascendental numbers \cite{zhuang} are possible values in the former. However realizing these values in groups to which Theorem \ref{thm:spectrum} applies is harder: one needs to perform a universal (or almost universal) central extension to annihilate second homology, while ensuring that the value of $\scl$ is preserved along the way.

The following is a general setup in which this construction goes through:

\begin{proposition}[Heuer--L\"oh {\cite[Section 3.3]{spectrum2}}]
\label{prop:ExtensionSimplVol}

Let $G$ be a group of orientation-preserving homeomorphisms of the circle, and let $\overline{G}$ be its total lift to the real line. Suppose that $G$ is finitely presented, it admits no unbounded quasimorphisms, and $\HH^2_b(G)$ is one-dimensional, spanned by the Euler class. Then there exists a central extension $E$ of $G$ with the following properties:
\begin{enumerate}
    \item $E$ is finitely presented;
    \item $\HH_2(E; \mathbf{R}) = 0$;
    \item $\scl(E') = \scl(\overline{G}')$ as subset of $\mathbf{R}$.
    \end{enumerate}
\end{proposition}

\begin{remark}
\label{rem:checkingHL}

Since $G$ is finitely presented, the hypothesis that it has no unbounded quasimorphisms can be replaced by: $G$ has finite abelianization, and the Euler class is not exact, i.e. its image in $\HH^2(G)$ is non-trivial. Using again finite presentability and the universal coefficient theorem, the latter condition is equivalent to the fact that the central extension $1 \to \mathbf{Z} \to \overline{G} \to G \to 1$ does not virtually split.
\end{remark}

In particular this theorem applies to the groups we considered in Theorem \ref{thm:T}. We immediately obtain:

\begin{corollary}
\label{cor:sv}

Let $G$ be as in Theorem \ref{thm:T}. Suppose that $G$ is finitely presented and has finite abelianization. Then for every $g \in \overline{G}'$ there exists a $4$-manifold $M$ such that $\|M\| = 48 \cdot \scl(g)$.
\end{corollary}

Theorem \ref{thm:sv} is a special case of this corollary:

\begin{proof}[Proof of Theorem \ref{thm:sv}]
Let $G = T_\tau$ be the golden ratio Thompson group. We know that $G$ is finitely presented (indeed type $F_{\infty}$, by Proposition \ref{prop:Tfin}) and has finite abelianization (from Proposition \ref{prop:index2T}). 
By Lemma \ref{lem:transT} and Theorem \ref{thm:T}, we know that $\HH^2_b(G)$ is one-dimensional and is spanned by the Euler class. Finally, the Euler class is not exact, since the central extension $1 \to \mathbf{Z} \to \overline{G} \to G \to 1$ does not virtually split, by Lemma \ref{lemma: lift}; thus the conditions of Proposition \ref{prop:ExtensionSimplVol} are satisfied (alternatively, one can show that $G'$ is uniformly perfect, and thus $G$ has no unbounded quasimorphisms \cite{algebraicscl}).

Moreover, $\scl(\overline{G}')$ contains all of $\mathbf{Z}[\tau]$, where $\tau = \frac{\sqrt{5}-1}{2}$ is the small golden ratio (Theorem \ref{thm:scl2}). Therefore by Proposition \ref{prop:ExtensionSimplVol}, there exists a $4$-manifold $M$ such that $\|M\| = 48 \cdot \tau$, which is algebraic and irrational.
Indeed, for each $\lambda \in \mathbf{Z}[\tau]$, we obtain a $4$-manifold $M$ such that $\|M\| = 48 \cdot \lambda$.
\end{proof}

In the same way, we obtain new transcendental values in the spectrum of simplicial volume:

\begin{example}
Stein--Thompson groups to which Theorem \ref{thm:T} applies (see \cite{monsters}) produce irrational values of stable commutator length in the same way \cite{zhuang}. For example, if $G = T_{2, 3}$, then $G$ is finitely presented, has no unbounded quasimorphisms, and there exist elements in its lift whose $\scl$ equals $\frac{\log(3)}{\log(2)}$. Then Corollary \ref{cor:sv} produces a $4$-manifold such that $\|M\| = 48 \cdot \frac{\log(3)}{\log(2)}$.
This confirms a conditional result of Heuer and L\"oh \cite[Corollary A.12]{spectrumv1}.
\end{example}

More generally, Corollary \ref{cor:sv} provides a way to translate many concrete examples of values of $\scl$ to values of simplicial volume. We hope that this approach will be useful for showing that the full spectra of $\scl$ of finitely presented groups and simplicial volume of $4$-manifolds coincide. It is conjectured \cite[Question 1.1]{spectrum2} that both spectra coincide with the set of right-computable numbers: this is known to be the full spectrum of $\scl$ of recursively presented groups \cite{heuer}.

%\pagebreak

\footnotesize

\bibliographystyle{abbrv}
\bibliography{references}

\normalsize

\noindent{\textsc{Department of Mathematics, ETH Z\"urich, Switzerland}}

\noindent{\textit{E-mail address:} \texttt{francesco.fournier@math.ethz.ch}} \\

\noindent{\textsc{Department of Mathematics,
University of \Hawaii at \Manoa.}}

\noindent{\textit{E-mail address:} \texttt{lodha@hawaii.edu}} \\

\end{document}